\documentclass{article}
\usepackage{amsopn,amstext,amsbsy,amsmath,amscd,amsthm,latexsym,amsfonts}
\usepackage{geometry} % packet för att ändra margialbredd
\usepackage{graphicx}
\usepackage[utf8]{inputenc} % så man kan skriva åäö
\usepackage{epstopdf}
\usepackage{amssymb}
\usepackage{graphics}
\usepackage{pgfplots}
\usepackage[titletoc,title]{appendix}
\usepackage[swedish,german,english]{babel}
\numberwithin{equation}{section}
\addto\captionsenglish{%
}
\usepackage[labelsep=endash]{caption}
\usepackage[all]{xy} % för att rita diagram
\allowdisplaybreaks[1]
\theoremstyle{plain}
\newtheorem{theorem}                 {\bf Theorem}

\newtheorem{lemma}        [theorem]  {\bf Lemma}

\title{Internal gravity-capillary solitary waves in finite depth}
\author{Dag Nilsson}
\theoremstyle{definition}

\DeclareMathOperator{\sgn}{sgn}
\newcommand\norm[1]{\left\lVert#1\right\rVert}

\newcommand{\abs}[1]{\lvert#1\rvert}
\begin{document}
\maketitle
\begin{abstract}
We consider two-dimensional inviscid irrotational flow in a two layer fluid under the effects of gravity and interfacial tension. The upper fluid is bounded above by a rigid lid, and the lower fluid is bounded below by a rigid bottom. We use a spatial dynamics approach and formulate the steady Euler equations as a Hamiltonian system, where we consider the unbounded horizontal coordinate $x$ as a time-like coordinate. The linearization of the Hamiltonian system is studied, and bifurcation curves in the $(\beta,\alpha)$-plane are obtained, where $\alpha$ and $\beta$ are two parameters. The curves depend on two additional parameters $\rho$ and $h$, where $\rho$ is the ratio of the densities and $h$ is the ratio of the fluid depths. However, the bifurcation diagram is found to be qualitatively the same as for surface waves. In particular we find that a Hamiltonian-Hopf bifurcation, Hamiltonian real 1:1 resonance and a Hamiltonian $0^2$-resonance occur for certain values of $(\beta,\alpha)$. Of particular interest are solitary wave solutions of the Euler equations. Such solutions correspond to homoclinic solutions of the Hamiltonian system. We investigate the parameter regimes where the Hamiltonian-Hopf bifurcation and the Hamiltonian real 1:1 resonance occur. In both these cases we perform a center manifold reduction of the Hamiltonian system and show that homoclinic solutions of the reduced system exist. In contrast to the case of surface waves we find parameter values $\rho$ and $h$ for which the leading order nonlinear term in the reduced system vanishes. We make a detailed analysis of this phenomenon in the case of the real 1:1 resonance. We also briefly consider the Hamiltonian $0^2$-resonance and recover the results found by Kirrmann \cite{KM}.
\end{abstract}
%\thispagestyle{empty}
%\centerline {\bf\Large Abstract}
\section{Introduction}
\subsection{Internal waves}
\begin{figure}\centering
\includegraphics{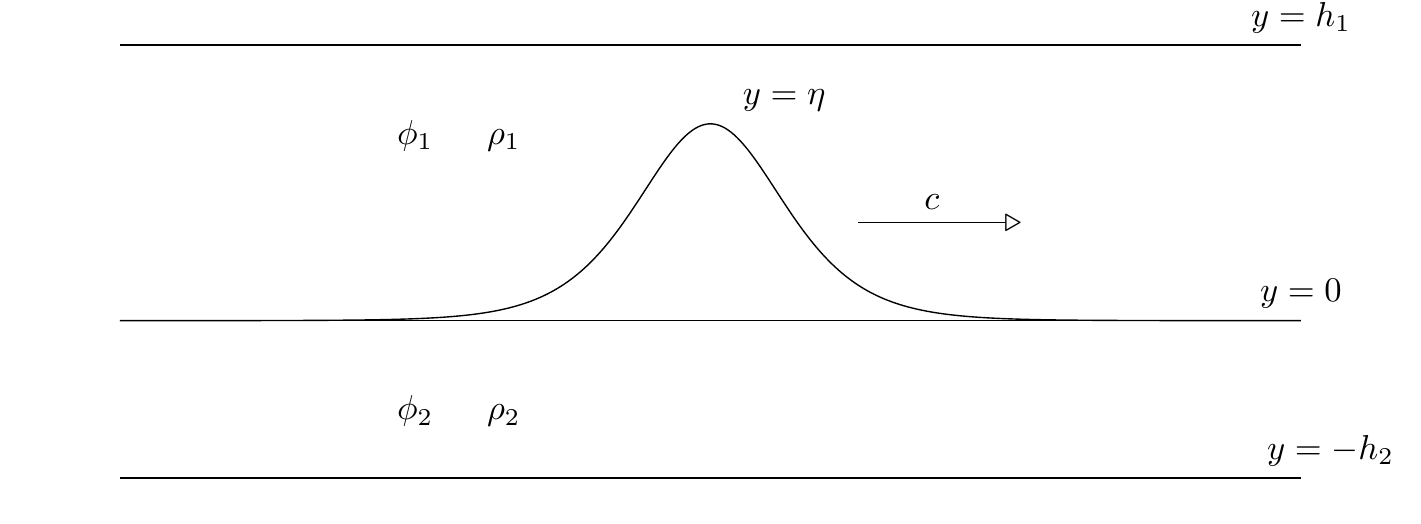}
\caption{Symmetric solitary internal wave of elevation}
\end{figure}
Internal waves are waves which propagate along the interface of two immiscible fluids of different density. In this paper we will study two-dimensional steady solitary internal waves in finite depth, under the influence of gravity and interfacial tension. These are waves which travel with constant speed $c$ in some distinguished direction, which we will assume to be the positive $x$-direction, and whose wave profile tends to $0$ as $\abs{x}\rightarrow\infty$. The flow is assumed to be inviscid and irrotational, and the density of each layer is assumed to be constant. In addition we assume that the upper fluid is bounded above by a rigid horizontal lid and the lower fluid is bounded below by a rigid horizontal bottom. The two fluids are separated by an interface $\eta$ in the infinite strip $\{(x,y)\in \mathbb{R}^2\text{ }\vert\text{ }-h_2\leq y \leq h_1\}$, where $h_1,h_2$ are positive numbers. Denote by $\rho_1$ the density of the upper fluid and $\rho_2$ the density of the lower fluid. We assume that $\rho_2>\rho_1$.
Denote by $\phi_1,\phi_2$ the velocity potentials of the upper and lower fluid respectively. In a moving frame of reference, the velocity potentials satisfy the equations
\begin{align}
&\Delta\phi_1=0 \quad \text{for } \eta<y<h_1 \label{e1},\\
&\Delta\phi_2=0 \quad \text{for } -h_2<y<\eta \label{e_2},
\end{align}
with boundary conditions
\begin{align}
&\phi_{1y}=0 && \text{on } y=h_1\label{e3},\\
&\phi_{2y}=0 && \text{on } y=-h_2\label{e4},\\
&-c\eta_x=\phi_{1y}-\eta_x\phi_{1x} && \text{on } y=\eta \label{e5},\\
&-c\eta_x=\phi_{2y}-\eta_x\phi_{2x} && \text{on } y=\eta \label{e6},\\
&\rho_2(-c\phi_{2x}+\frac{1}{2}\vert \nabla\phi_2\vert^2+g\eta )-\rho_1(-c\phi_{1x}+\frac{1}{2}\vert\nabla\phi_1\vert^2+g\eta )=\sigma \left(\frac{\eta_x}{\sqrt{1+\eta_x^2}}\right)_x&&\text{on } y=\eta \label{e7},%se (1.30) i Johnson
\end{align}
where $\sigma$ is the coefficient of interfacial tension. Note that for $\rho_1=0$ we obtain the surface wave problem. A solitary wave is a solution of \eqref{e1}--\eqref{e7} together with the asymptotic conditions
\begin{equation*} 
\lim_{\abs{x}\rightarrow \infty}\eta=0,\ \lim_{\abs{x}\rightarrow\infty}(\phi_{ix},\phi_{iy})=(0,0),\ i=1,2.
\end{equation*}
We will construct solitary-wave solutions to these equations using the method of spatial dynamics. The idea, which goes back to Kirchgässner \cite{K82}, is to formulate the water wave problem as an evolution equation
\begin{equation}
u_x=Ku+F(u),\label{cmeq}
\end{equation}
where $K$ is a linear operator and $F(u)=\mathcal{O}(u^2)$. The variable $x$ is an unbounded spatial coordinate which is treated as a time-like coordinate. The evolution equation is ill-posed, but bounded solutions can be studied using a parameter dependent version of the center manifold theorem. This theorem is used to obtain a finite-dimensional system on the center manifold, which is locally equivalent to the original system \eqref{cmeq}. The dimension of the center manifold is equal to the number of purely imaginary eigenvalues of the operator $K$. It is therefore necessary to study the spectrum of $K$.
\subsection{Previous work on surface waves}
We give here a short review of some of the results that have been obtained for surface waves.
In the water wave problem the spectrum of $K$ depends on two parameters $\alpha$ and $\beta$, where $\alpha=\frac{gh}{c^2}$, $\beta=\frac{\sigma}{\rho hc^2}$. Here $c$ is the speed of the solitary wave, $\sigma$ is the coefficient of surface tension, $h$ is the water depth, $\rho$ is the density of the fluid and $g$ is the gravitational constant\footnote{When considering surface waves as internal waves with $\rho_1=0$ we should use $h_2$ instead of $h$ and $\rho_2$ instead of $\rho$ in the definition of $\alpha$ and $\beta$.}. Kirchgässner \cite{KI} identified four critical curves in the $(\beta,\alpha)$ plane where the spectrum of $K$ changes. He also studied the region $\alpha=1+\epsilon$ and $\beta>\frac{1}{3}$, where $0<\epsilon\ll1$, and showed that a Hamiltonian $0^2$ resonance occurs for $\epsilon=0$. For $\alpha=1$ and $\beta>\frac{1}{3}$ the imaginary part of the spectrum consists of the $0$-eigenvalue, which is of algebraic multiplicity two. This means that the center manifold is two-dimensional in this case. There exist a homoclinic solution for the two-dimensional reduced system when $\epsilon=0$, and this solution persists for $\epsilon>0$. The corresponding wave profile is a solitary wave of depression. This region was also studied in \cite{AK89}, \cite{SA} using different techniques.

Iooss and Kirchgässner \cite{IK} investigated a critical curve where a Hamiltonian-Hopf bifurcation occurs (called a 1:1 resonance for systems that do not have a Hamiltonian structure). For $\alpha$ and $\beta$ belonging to this curve, the imaginary part of the spectrum of $K$ is equal to $\pm ik$, and these are eigenvalues of algebraic multiplicity two, which implies that the corresponding center manifold is four-dimensional. There exist two families of homoclinic solutions of the reduced system, if certain coefficients have the correct signs. These solutions correspond to periodic wave trains with exponentially decaying envelopes and are sometimes called bright solitary waves. The term bright solitary wave comes from its connection with the nonlinear Schrödinger equation (see \eqref{nlsintro}). This will be discussed more below, in connection with the Hamiltonian-Hopf bifurcation for internal waves, were we also consider dark solitary waves. In \cite{DI2} it was verified that these coefficients do have the correct signs. The authors of \cite{IP} showed that the solutions of the reduced equations persist for the full system, thereby proving existence of solitary waves in this parameter regime. Moreover, it was shown in \cite{BG} that there exists an infinite number of homoclinic solutions of the reduced equations. The corresponding wave profiles look like several of the solutions found in \cite{IK} glued together.  

Buffoni Groves and Toland \cite{BGT} studied a curve in the parameter plane where a Hamiltonian real 1:1 resonance takes place. When applying the center manifold theorem they obtained a four-dimensional system. This four-dimensional system is, when the bifurcation parameter is set to zero, equivalent to the fourth order ODE
\begin{equation}
\partial_X^4u-2(1+\delta)\partial_X^2u+u-u^2=0.\label{bgt}
\end{equation}
It was shown in \cite{BCT} that this equation has a solution which is unique up to translation for $\delta=0$, and for $\delta<0$ there exists an infinite number of solutions which resemble multiple copies of the first one. That these solutions persist for the full system was shown in \cite{BGT}.

Iooss and Kirchgässner \cite{IK2} also studied the case when the surface tension is small. In particular they considered the case when $\alpha=1-\delta$ and $\beta<\frac{1}{3}$. In this region a Hamiltonian $0^2$ik resonance occurs. The authors showed that there exists a family of solutions which correspond to generalized solitary waves. These are waves which resemble solitary waves, but the wave profile does not tend to zero at infinity. Instead there are oscillatory ripples at infinity. We mention that the case when $\beta=0$ also can be studied using spatial dynamics methods; see \cite{Dias2003}, \cite{GW08}, \cite{Mielke1} and references therein. We also mention \cite{GW07} where the authors apply spatial dynamics methods to study gravity-capillary solitary waves with an arbitrary distribution of vorticity.

\subsection{Previous work on internal waves}
Consider \eqref{e1}--\eqref{e7}. When these equations are nondimensionalized (see \eqref{e15}--\eqref{e21}) four parameters $\alpha,\ \beta,\ \rho$ and $h$ emerge. Here $\alpha=gh_1(1-\rho)/c^2$, $\beta=\sigma/(h_1\rho_2 c^2)$,  $\rho=\rho_1/\rho_2$ and $h=\frac{h_2}{h_1}$. As we shall see later in section \ref{parameterregimes} the bifurcation diagram is qualitatively the same as in the surface wave case. The difference is that the bifurcation curves now depend upon $\rho$ and $h$ (see figure \ref{fig2}).

There are several results concerning steady internal solitary waves. Amick and Turner \cite{AT} proved the existence of solitary waves in the finite depth case, with $\beta=0$. In particular, they found solitary waves of elevation when $\rho<\frac{1}{h^2}$ and solitary waves of depression when $\rho>\frac{1}{h^2}$. The case when $\rho\approx\frac{1}{h^2}$ has also been studied. For example, in \cite{AT2} the authors found small amplitude solitary waves using spatial dynamics techniques and the center manifold theorem.

Kirrmann \cite{KM} considered the case when the coefficient of interfacial tension is nonzero using a spatial dynamics formulation. In particular, he considered the case when $\beta>\frac{\rho+h}{3}$ and $\alpha$ is greater than but close to $\rho+1/h$ (see region III in figure \ref{fig2}). In this parameter-regime he found solitary waves of elevation and depression and again studied the case when $\rho \approx 1/h^2$. This case was also studied in \cite{SC} using different methods.

The above solitary waves appear in a $0^2$ bifurcation at $\alpha=\rho+1/h$. In $\cite{DI}$ the authors considered Hamiltonian-Hopf bifurcation for internal waves in the infinite depth case. The authors formulate the problem as an equation of the form \eqref{cmeq}. One of the main difficulties here is that the real line is contained in the spectrum of the linear operator $K$. As a consequence of this the center manifold theorem cannot be applied. However it is still possible to proceed formally and in doing so they find a critical value $\rho_c$ of $\rho$ such that when $\rho<\rho_c$ they obtain bright solitary wave solutions, and when $\rho>\rho_c$ they obtain dark solitary wave solutions.

\subsection{Outline of this paper}\label{contribution}
In section \ref{hfeu} we obtain a Hamiltonian formulation of the traveling wave problem. This is done by first writing it as a variational problem and identifying an action integral. The Hamiltonian formulation is then obtained via a Legendre transform.
In section \ref{parameterregimes} we study the spectrum of the linear part of the Hamiltonan system and obtain a bifurcation diagram (see figure \ref{fig2}). In particular we identify a Hamiltonian-Hopf bifurcation, a Hamiltonian real 1:1 resonance and a $0^2$ resonance.
A change of variables is found in section \ref{changeofvariables}, such that the boundary conditions in the Hamiltonian system becomes linear. This is needed when applying the center manifold theorem. 
Section \ref{cmtsection} consists of a statement of the version of the center manifold theorem which we are using, and also verification of the hypotheses of the theorem.
In section \ref{hamhopf} we consider the Hamiltonian-Hopf bifurcation. The solutions of the reduced truncated equation satisfy the nonlinear Schrödinger equation:
\begin{equation}
\tilde{A}_{\tilde{x}\tilde{x}}+\sgn(\delta)c_1\tilde{A}+2c_3\tilde{A}\abs{\tilde{A}}^2=0,\label{nlsintro}
\end{equation}
where $\delta$ is a bifurcation parameter. The coefficient $c_1$ is found to be negative, however the sign of $c_3$ depends upon the wavenumber $k$, $\rho$ and $h$. When $c_3>0$  the equation is focusing, and we obtain a family of bright solitary waves, two of which persist when the remainder terms are included, similar to the surface wave case \cite{IK} and the internal wave case for infinite depth \cite{DI}. When $c_3<0$ the equation is defocusing,  and we find dark solitary waves which were also found in the infinite-depth case \cite{DI}.

The Hamiltonian real 1:1 resonance is considered in section \ref{11}. Equation \eqref{bgt} is obtained when assuming that $\abs{\rho-\frac{1}{h^2}}>0$ and we find solutions similar to those in \cite{BGT}. The difference is that the waves are of elevation when $\rho-1/h^2>0$ and of depression when $\rho-1/h^2<0$, whereas in the surface wave case only depression waves are found. We also mention here that equation \eqref{bgt} was derived formally in \cite{SC} for internal waves. When $\rho=1/h^2$, the coefficient of the quadratic term in \eqref{bgt} is equal to $0$. It is therefore necessary to include higher order terms in the Taylor expansion of the reduced Hamiltonian. When the bifurcation parameter $\epsilon$ is set to zero, we obtain the equation:
\begin{equation}
\partial_X^4 u -2(1+\delta)\partial_X^2 u +u-c\kappa u^2-u^3=0,\label{truncatedeq4}
\end{equation}
where $c>0$ is a constant and $\kappa$ is allowed to be either positive or negative. We mention here that equation \eqref{truncatedeq4} is written down in \cite{LD}. However the authors do not consider this equation except when discussing experimental setups (see the discussion at the end of this section). See also \cite{SC} where the authors formally derive another fourth-order ODE for this case. By using the same arguments as in \cite{BCT}, we are able to show that homoclinic solutions of \eqref{truncatedeq4}, when $\delta=\kappa=0$, correspond to a transversal intersection in the zero energy manifold. Using the stable manifold theorem we may then conclude that these solutions persist when $\epsilon>0$, $\delta<0$ and $\kappa$ are sufficiently small. We actually find two distinct solutions here since if $u$ solves \eqref{truncatedeq4} for $\kappa=0$, then $-u$ is a solution as well. Moreover, the spatial dynamics formulation \eqref{cmeq} is Hamiltonian in our case. This allows us to apply the theory from \cite{DV}, which says that for $\delta<0$ sufficiently small there exists a countable family of solutions, which resemble multiple copies of the primary homoclinic orbit. In our case we actually get two such families, one family corresponding to $u$ and the other one to $-u$.

In section \ref{02} the Hamiltonian $0^2$-resonance is considered. Here we recover the results found in \cite{KM}, that is we find solitary waves of elevation when $\rho>1/h^2$, of depression when $\rho<1/h^2$ and of elevation and of depression when $\rho-1/h^2$ is small.%skriva mer här? tag med ekvationerna? 

Finally, we make a brief remark on the physical parameter values needed in order to observe the waves studied in this paper.
In \cite{LD} the authors suggest experimental setups where the waves from the Hamiltonian-Hopf bifurcation and the Hamiltonian real 1:1 resonance could be observed. In these setups the lower fluid is water and the upper is a mixture of silicone oil and 1-2-3-4 tetrahydronaphtalene. 
We also refer \cite{JA} for a collection of images of internal waves in the ocean together with physical data.
\section{Spatial dynamics formulation of the traveling wave problem}\label{hfeu}
We introduce the non-dimensional variables
\begin{equation*}
(x^\prime , y^\prime )=\frac{1}{h_1}(x,y),\quad \eta^\prime(x^\prime)=\frac{\eta(x)}{h_1},\quad \phi_i^\prime(x^\prime,y^\prime)=\frac{\phi_i(x,y)}{h_1c},\ i=1,2.
\end{equation*}
This gives us the following system of equations, where we for notational simplicity drop the $\prime$:
\begin{align}
&\Delta\phi_1=0 \quad \text{for } \eta<y<1, \label{e15}\\
&\Delta\phi_2=0 \quad \text{for } -h<y<\eta, \label{e16}
\end{align}
where $h=\frac{h_2}{h_1}$, and with boundary conditions
\begin{align}
&\phi_{1y}=0  &&\text{on } y=1,\label{e17}\\
&\phi_{2y}=0  && \text{on }y=-h,\label{e18}\\
&-\eta_x=\phi_{1y}-\eta_x\phi_{1x} && \text{on } y=\eta, \label{e19}\\
&-\eta_x=\phi_{2y}-\eta_x\phi_{2x} && \text{on } y=\eta, \label{e20}\\
&-\phi_{2x}+\frac{1}{2}\vert\nabla\phi_2\vert^2-\rho(-\phi_{1x}+\frac{1}{2}\vert\nabla\phi_1\vert^2)+\alpha\eta= \beta\left(\frac{\eta_x}{\sqrt{1+\eta_x^2}}\right)_x && \text{on } y=\eta,\label{e21}
\end{align}
where $\rho=\rho_1/\rho_2$, $\alpha=gh_1(1-\rho)/c^2$ and $\beta=\sigma/(h_1\rho_2 c^2)$.

The domains of $\phi_i$ depend upon the unknown $\eta$. In order to obtain a fixed domain we introduce the following change of variables. Let
\begin{equation*}
z(x,y)=\begin{cases}
\frac{y-1}{\eta(x)-1},& \eta<y<1,\\
 \frac{y+h}{\eta(x)+h},&-h<y<\eta.
\end{cases}
\end{equation*}
Clearly $z(x,y)\in (0,1)$ for all $x\in \mathbb{R}$ and all $y\in (-h,\eta(x))\cup(\eta(x),1)$. Let $\tilde{\phi}_1(x,z)=\phi_1(x,y)$ for $\eta(x)<y<1$ and let $\tilde{\phi}_2(x,z)=\phi_2(x,y)$ for $-h<y<\eta(x)$. The system \eqref{e15}--\eqref{e21} now becomes
\begin{align}
&\tilde{\phi}_{1xx}-\frac{2z\eta_x}{\eta-1}\tilde{\phi}_{1xz}-\frac{z\eta_{xx}}{\eta-1}\tilde{\phi}_{1z}+\frac{2z\eta_x^2}{(\eta-1)^2}\tilde{\phi}_{1z}+\frac{z^2\eta_x^2}{(\eta-1)^2}\tilde{\phi}_{1zz}+\frac{1}{(\eta-1)^2}\tilde{\phi}_{1zz}=0,&& 0<z<1\label{e22}\\
&\tilde{\phi}_{2xx}-\frac{2z\eta_x}{\eta+h}\tilde{\phi}_{2xz}-\frac{z\eta_{xx}}{\eta+h}\tilde{\phi}_{2z}+\frac{2z\eta_x^2}{(\eta+h)^2}\tilde{\phi}_{2z}+\frac{z^2\eta_x^2}{(\eta+h)^2}\tilde{\phi}_{2zz}+\frac{1}{(\eta+h)^2}\tilde{\phi}_{2zz}=0,&& 0<z<1\label{e23}
\end{align}
with boundary conditions
\begin{align}
\tilde{\phi}_{1z}&=0 && \text{on } z=0,\label{e24}\\
\tilde{\phi}_{2z}&=0 && \text{on } z=0,\label{e25}\\
-\eta_x&=\frac{1}{\eta-1}\tilde{\phi}_{1z}-\eta_x\tilde{\phi}_{1x}+\frac{\eta_x^2}{\eta-1}\tilde{\phi}_{1z}&& \text{ on }z=1, \label{e26}\\
-\eta_x&=\frac{1}{\eta+h}\tilde{\phi}_{2z}-\eta_x\tilde{\phi}_{2x}+\frac{\eta_x^2}{\eta+h}\tilde{\phi}_{2z}&& \text{ on }z=1, \label{e27}\\
&-\phi_{2x}+\frac{\eta_x\phi_{2z}}{\eta+h}+\frac{1}{2}\bigg(\phi_{2x}-\frac{\eta_x\phi_{2z}}{\eta+h}\bigg)^2+\frac{\phi_{2z}^2}{2(\eta+h)^2}&&\nonumber\\
&-\rho\Bigg[-\phi_{1x}+\frac{\eta_x\phi_{1z}}{\eta-1}+\frac{1}{2}\bigg(\phi_{1x}-\frac{\eta_x\phi_{1z}}{\eta-1}\bigg)^2+\frac{\phi_{1z}^2}{2(\eta-1)^2}\Bigg]+\alpha\eta&&\nonumber\\
&= \beta\left(\frac{\eta_x}{\sqrt{1+\eta_x^2}}\right)_x&&\text{on } z=1.\label{e28}
\end{align}
From now on we drop the $\sim$ notation. 

We write down the energy $E$ and momentum $P$ associated with this system
\begin{align*}
E&=\frac{h_1^2\rho_1c^2}{2}\int_{\mathbb{R}}\int_0^1\left(\left(\phi_{1x}-\frac{z\eta_x}{\eta-1}\phi_{1z}\right)^2+\frac{\phi_{1z}^2}{(\eta-1)^2}\right)(1-\eta)\ dz\ dx\\
&\quad +\frac{h_1^2\rho_2c^2}{2}\int_{\mathbb{R}}\int_0^1\left(\left(\phi_{2x}-\frac{z\eta_x}{\eta+h}\phi_{2z}\right)^2+\frac{\phi_{2z}^2}{(\eta+h)^2}\right)(\eta+h)\ dz\ dx\\
&\quad +\frac{g(\rho_2-\rho_1)h_1^3}{2}\int_{\mathbb{R}}\eta^2dx+\sigma h_1\int_{\mathbb{R}}\left(\sqrt{1+\eta_x^2}-1\right)\ dx,
\end{align*}
and
\begin{align*}
P=\rho_1h_1^2c\int_\mathbb{R}\int_0^1(\phi_{1x}-\frac{z\eta_x}{\eta-1}\phi_{1z})(1-\eta)\ dz\ dx+\rho_2h_1^2c\int_\mathbb{R}\int_0^1(\phi_{2x}-\frac{z\eta_x}{\eta+h}\phi_{2z})(\eta+h)\ dz\ dx.
\end{align*}
The solutions we are interested in are critical points of the functional $E-cP$, and
\begin{align*}
E-cP&%=h_1^2\rho_2c^2\left[\frac{\rho}{2}\int_{\mathbb{R}}\int_0^1\left((\phi_{1x}-\frac{z\eta_x}{\eta-1}\phi_{1z})^2+\frac{\phi_{1z}^2}{(\eta-1)^2}\right)(1-\eta)dzdx\right.\\
%&+\frac{1}{2}\int_{\mathbb{R}}\int_0^1\left((\phi_{2x}-\frac{z\eta_x}{\eta+h}\phi_{2z})^2+\frac{\phi_{2z}^2}{(\eta+h)^2}\right)(\eta+h)dzdx\\
%&\left.+\frac{g(1-\rho)h_1}{2c^2}\int_\mathbb{R}\eta^2dx+\beta\int_\mathbb{R}(\sqrt{1+\eta_x^2}-1)dx-\rho\int_\mathbb{R}\int_0^1(\phi_{1x}-\frac{z\eta_x}{\eta-1}\phi_{1z})(1-\eta)dzdx\right.\\
%&\left.-\int_\mathbb{R}\int_0^1(\phi_{2x}-\frac{z\eta_x}{\eta+h}\phi_{2z})(\eta+h)dzdx\right]\\
%&=h_1^2\rho_2c^2\left[\frac{\rho}{2}\int_\mathbb{R}\int_0^1((\phi_{1x}-\frac{z\eta_x}{\eta-1}\phi_{1z}-1)^2-1+\frac{\phi_{1z}^2}{(\eta-1)^2})(\eta-1)dzdx\right.\\
%&\left.+\frac{1}{2}\int_\mathbb{R}\int_0^1((\phi_{2x}-\frac{z\eta_x}{\eta+h}\phi_{2z}-1)^2-1+\frac{\phi_{2z}^2}{(\eta+h)^2})(\eta+h)dzdx\right.\\
%&\left.+\int_\mathbb{R}\frac{g(1-\rho)h_1}{2c^2}\eta^2dx+\beta\int_\mathbb{R}(\sqrt{1+\eta_x^2}-1)dx\right]\\
=h_1^2\rho_2c^2\left[\frac{\rho}{2}\int_\mathbb{R}\int_0^1\left(\left(\phi_{1x}-\frac{z\eta_x}{\eta-1}\phi_{1z}-1\right)^2+\frac{\phi_{1z}^2}{(\eta-1)^2}\right)(1-\eta)\ dz\ dx\right.\\
&\left.\quad +\frac{1}{2}\int_\mathbb{R}\int_0^1\left(\left(\phi_{2x}-\frac{z\eta_x}{\eta+h}\phi_{2z}-1\right)^2+\frac{\phi_{2z}^2}{(\eta+h)^2}\right)(\eta+h)\ dz\ dx\right.\\
&\left.\quad +\int_\mathbb{R}\frac{\alpha}{2}\eta^2\ dx+\beta\int_\mathbb{R}\left(\sqrt{1+\eta_x^2}-1\right)\ dx-\frac{\rho}{2}\int_\mathbb{R}(1-\eta)\ dx-\frac{1}{2}\int_\mathbb{R}(\eta+h)\ dx\right].\\
\end{align*}
We regard this as an action integral, where the Lagrangian is determined by
\begin{align*}
L(\eta,\eta_x,\phi_1,\phi_{1x},\phi_2,\phi_{2x})&=\frac{\rho}{2}\int_0^1\left(\left(\phi_{1x}-\frac{z\eta_x}{\eta-1}\phi_{1z}-1\right)^2+\frac{\phi_{1z}^2}{(\eta-1)^2}\right)(1-\eta)\ dz\\
&\quad+\frac{1}{2}\int_0^1\left(\left(\phi_{2x}-\frac{z\eta_x}{\eta+h}\phi_{2z}-1\right)^2+\frac{\phi_{2z}^2}{(\eta+h)^2}\right)(\eta+h)\ dz\\
&\quad+\frac{\alpha}{2}\eta^2+\beta\left(\sqrt{1+\eta_x^2}-1\right)-\frac{\rho}{2}(1-\eta)-\frac{1}{2}(\eta+h).
\end{align*}
A Hamiltonian formulation of \eqref{e22}--\eqref{e28} is obtained via the Legendre transform
\begin{align*}
\psi_1 &=\frac{\delta L}{\delta\phi_{1x}}=\rho(1-\eta)\left(\phi_{1x}-\frac{z\eta_x\phi_{1z}}{\eta-1}-1\right),\\
\psi_2 &=\frac{\delta L}{\delta \phi_{2x}}=(\eta+h)\left(\phi_{2x}-\frac{z\eta_x\phi_{2x}}{\eta+h}-1\right),\\
\omega &=\frac{\delta L}{\delta \eta_x}=\int_0^1 \rho z\phi_{1z}\left(\phi_{1x}-\frac{z\eta_x\phi_{1z}}{\eta-1}-1\right)dz-\int_0^1z\phi_{2z}\left(\phi_{2x}-\frac{z\eta_x\phi_{2z}}{\eta+h}-1\right)dz+\frac{\beta\eta_x}{\sqrt{1+\eta_x^2}}\\
& =-\int_0^1\frac{z\phi_{1z}\psi_1}{\eta-1}\ dz-\int_0^1 \frac{z\phi_{2z}\psi_2}{\eta+h}\ dz+\frac{\beta\eta_x}{\sqrt{1+\eta_x^2}}.
\end{align*}
The Hamiltonian $\widehat{\mathcal{H}}$ is defined by
\begin{align}
\widehat{\mathcal{H}}(\eta,\omega,\phi_1,\psi_1,\phi_2,\psi_2)&=\int_0^1\psi_1\phi_{1x}dz+\int_0^1\psi_2\phi_{2x}dz+\omega\eta_x -L(\eta,\omega,\phi_1,\psi_1,\phi_2,\psi_2)\nonumber\\
&=\int_0^1\frac{1}{2\rho(1-\eta)}\left((\psi_1+(1-\eta)\rho)^2-\rho^2\phi_{1z}^2\right)\ dz\nonumber\\
&\quad+\int_0^1\frac{1}{2(\eta+h)}\left((\psi_2+(\eta+h))^2-\phi_{2z}^2\right)\ dz-\sqrt{\beta^2-\bar{\omega}^2}+\beta\nonumber\\
&\quad-\frac{\alpha}{2}\eta^2,\label{hamiltonian}
\end{align}
where
\begin{equation*} 
\bar{\omega}=\omega+\int_0^1\frac{z\phi_{1z}\psi_1}{\eta-1}\ dz+\int_0^1 \frac{z\phi_{2z}\psi_2}{\eta+h}\ dz.
\end{equation*} 
For $s\geq 0$, define 
\begin{equation*}
X_s=\mathbb{R}\times\mathbb{R}\times H^{s+1}(0,1)\times H^s( 0,1)\times H^{s+1}(0,1)\times H^s( 0,1).
\end{equation*}
Let $\widehat{M}=X_0$, $m\in\widehat{M}$ and let $v=(\eta,\omega,\phi_1,\psi_1,\phi_2,\psi_2)\in T_m\widehat{M}$. On $T_m\widehat{M}\times T_m\widehat{M}$ we define the position independent symplectic form
\begin{align}
\widehat{\Omega}(v,v^*)&=\omega^*\eta-\eta^*\omega +\int_0^1(\psi_1^*\phi_1-\phi_1^*\psi_1)\ dz+\int_0^1(\psi_2^*\phi_2-\phi_2^*\psi_2)\ dz\label{omega1}.
\end{align}
Observe that $(\widehat{M},\widehat{\Omega})$ is a symplectic manifold. The set
\begin{equation*}
\widehat{N}=\{m\in \widehat{M}\text{ }:\text{ } \vert \bar{\omega}\vert<\beta,\text{ } -h<\eta<1\}
\end{equation*}
is a manifold domain of $\widehat{M}$ and $\widehat{\mathcal{H}}\in C^\infty(N,\mathbb{R})$. The triple $(\widehat{M},\widehat{\mathcal{H}},\widehat{\Omega})$ is therefore a Hamiltonian system. The Hamiltonian vector field $v_{\widehat{\mathcal{H}}}$ is defined by
\begin{equation*}
\text{Dom}(v_{\widehat{\mathcal{H}}}):=\{m\in \text{Dom}(\widehat{\mathcal{H}})\ \vert\text{ } \exists (v_{\widehat{\mathcal{H}}})_m\in T_m\widehat{M} \text{ such that } d\widehat{\mathcal{H}}[m](v_m^*)=\widehat{\Omega}((v_{\widehat{\mathcal{H}}})_m,v_m^*) \text{ } \forall v_m^*\in T_m\widehat{M}\},
\end{equation*} 
and Hamilton's equation is given by
\begin{equation}
\dot{\gamma}(x)=(v_{\widehat{\mathcal{H}}})_{\gamma(x)}.\label{heq}
\end{equation}
We find that
\begin{align*}
d\widehat{\mathcal{H}}[m](v_m^*)&=\Bigg(\int_0^1\left[\frac{\rho}{2(1-\eta)^2}(\frac{\psi_1^2}{\rho^2}-\phi_{1z}^2)-\frac{\rho}{2}\right]\ dz+\int_0^1\left[\frac{-1}{2(\eta+h)^2}(\psi_2^2-\phi_{2z}^2)+\frac{1}{2}\right]\ dz\\
&\quad+\frac{\bar{\omega}}{\sqrt{\beta^2-\bar{\omega}^2}}\left[\frac{-1}{(\eta-1)^2}\int_0^1z\phi_{1z}\psi_1dz+\frac{-1}{(\eta+h)^2}\int_0^1z\phi_{2z}\psi_2dz\right]-\alpha\eta\Bigg)\eta^*\\
&\quad+\frac{\bar{\omega}}{\sqrt{\beta^2-\bar{\omega}^2}}\omega^*+\int_0^1\frac{1}{\eta-1}\left(\frac{\bar{\omega}}{\sqrt{\beta^2-\bar{\omega}^2}}z\psi_1+\rho\phi_{1z}\right)\phi_{1z}^*\ dz\\
&\quad+\int_0^1\frac{1}{\eta-1}\left(-\frac{\psi_1}{\rho}+(\eta-1)+\frac{\bar{\omega}}{\sqrt{\beta^2-\bar{\omega}^2}}z\phi_{1z}\right)\psi_1^*\ dz\\
&\quad+\int_0^1\frac{1}{\eta+h}\left(\frac{\bar{\omega}}{\sqrt{\beta^2-\bar{\omega}^2}}z\psi_2-\phi_{2z}\right)\phi_{2z}^*\ dz\\
&\quad+\int_0^1\frac{1}{\eta+h}\left(\psi_2+(\eta+h)+\frac{\bar{\omega}}{\sqrt{\beta^2-\bar{\omega}^2}}z\phi_{2z}\right)\psi_2^*\ dz,
\end{align*}
for $m\in \widehat{N}$ and $v_m^*=(\eta^*,\omega^*,\phi_1^*,\psi_1^*,\phi_2^*,\psi_2^*)$.
By comparing the above expression with \eqref{omega1}, we get that $(v_{\widehat{\mathcal{H}}})_m$ is given by
\begin{align*}
\eta_{\widehat{\mathcal{H}}}&=\frac{\bar{\omega}}{\sqrt{\beta^2-\bar{\omega}^2}},\\
\omega_{\widehat{\mathcal{H}}}&=\int_0^1\left[-\frac{\rho}{2(\eta-1)^2}\left(\frac{\psi_1^2}{\rho^2}-\phi_{1z}^2\right)+\frac{\rho}{2}\right]dz+\int_0^1\left[\frac{1}{2(\eta+h)^2}(\psi_2^2-\phi_{2z}^2)-\frac{1}{2}\right]dz\\
&\quad +\frac{\bar{\omega}}{\sqrt{\beta^2-\bar{\omega}^2}}\left(\int_0^1\frac{z\phi_{1z}\psi_1dz}{(\eta-1)^2}+\int_0^1\frac{z\phi_{2z}\psi_2dz}{(\eta+h)^2}\right)+\alpha\eta,\\
\phi_{1\widehat{\mathcal{H}}}&=\frac{1}{\eta-1}\left(-\frac{\psi_1}{\rho}+(\eta-1)+\frac{\bar{\omega}z\phi_{1z}}{\sqrt{\beta^2-\bar{\omega}^2}}\right),\\
\psi_{1\widehat{\mathcal{H}}}&=\frac{1}{\eta-1}\left(\frac{\bar{\omega}(z\psi_1)_z}{\sqrt{\beta^2-\bar{\omega}^2}}+\rho\phi_{1zz}\right),\\
\phi_{2\widehat{\mathcal{H}}}&=\frac{1}{\eta+h}\left(\psi_2+(\eta+h)+\frac{\bar{\omega}z\phi_{2z}}{\sqrt{\beta^2-\bar{\omega}^2}}\right),\\
\psi_{2\widehat{\mathcal{H}}}&=\frac{1}{\eta+h}\left(\frac{\bar{\omega}(z\psi_2)_z}{\sqrt{\beta^2-\bar{\omega}^2}}-\phi_{2zz}\right).
\end{align*}
The domain $\mathcal{D}(v_{\widehat{\mathcal{H}}})$ is given by the elements $m \in N\cap X_1$ which satisfy
\begin{align}
&\begin{cases}
\rho\phi_{1z}(1)=-\frac{\bar{\omega}\psi_1(1)}{\sqrt{\beta^2-\bar{\omega}^2}},\\
\phi_{1z}(0)=0,
\end{cases}\label{b1}\\
&\begin{cases}
\phi_{2z}(1)=\frac{\bar{\omega}\psi_2(1)}{\sqrt{\beta^2-\bar{\omega}^2}},\\
\phi_{2z}(0)=0.
\end{cases}\label{b2}
\end{align}
Let
\begin{equation*}
 Y_s=\mathbb{R}\times\mathbb{R}\times H^{s+1}(0,1)\times H_0^{s+1}(0,1)\times H^{s+1}(0,1)\times H_0^{s+1}(0,1),
\end{equation*}
where $H_0^{s+1}(0,1)=\{f\in H^{s+1}(0,1)\ \vert \ f(0)=f(1)=0\}, $  and let
\begin{align*}
M&=\{m \in \mathbb{R}^2\times H^1(0,1)^4\ : \ \Gamma_1(0)=\Gamma_2(0)=\int_0^1\tilde{\phi}_1\ dz=\int_0^1\tilde{\phi}_2\ dz= 0\},\\
\tilde{M}&= \{m\in Y_0\ :\ \int_0^1\tilde{\phi}_1\ dz=\int_0^1\tilde{\phi}_2\ dz= 0\},\\
\tilde{N}&=\{m\in \tilde{M}\ : \ \abs{\bar{\omega}}<\beta,\text{ } -h<\eta<1\},
\end{align*}
where $m=(\eta,\omega,\tilde{\phi}_1,\Gamma_1,\tilde{\phi}_2,\Gamma_2)$. The equation \eqref{heq} has an equilibrium point $(0,0,0,-\rho,0,-h)$ which can be moved to the origin by the translation $\psi_1\rightarrow \psi_1+ \rho,\ \psi_2\rightarrow \psi_2+h$. 
Moreover, it is convenient to make the change of variables;
\begin{align*}
\Gamma_1&=\int_0^z (\psi_1+\rho)\ ds,\\
\Gamma_2&=\int_0^z(\psi_2+h)\ ds ,\\
\tilde{\phi}_1&=\phi_{1}-\chi_1,\\
\tilde{\phi}_2&=\phi_{2}-\chi_2,\\
\chi_1&=\int_0^1\phi_1\ dz,\\
\chi_2&=\int_0^1\phi_2\ dz,
\end{align*} 
which maps $(\eta,\omega,\phi_1,\psi_1,\phi_2,\psi_2)\in \widehat{M}$ to $(\eta,\omega,\tilde{\phi}_1,\Gamma_1,\tilde{\phi}_2,\Gamma_2,\chi_1,\chi_2)\in M\times \mathbb{R}^2$.
The symplectic product $\widehat{\Omega}$ transform into
\begin{align*}
\Omega(v,v^*)&=\omega^*\eta-\eta^*\omega +\int_0^1\Gamma_{1z}^*\tilde{\phi}_1-\tilde{\phi}_1^*\Gamma_{1z} \ dz+\int_0^1\Gamma_{2z}^*\tilde{\phi}_2-\tilde{\phi}_2^*\Gamma_{2z} \ dz+\Gamma_1^*(1)\chi_1-\chi_1^*\Gamma_1(1)\\
&\quad +\Gamma_2^*(1)\chi_2-\chi_2^*\Gamma_2(1),
\end{align*}
 and the the Hamiltonian $\widehat{\mathcal{H}}$ into
\begin{align*} \mathcal{H}(\eta,w,\tilde{\phi}_1,\Gamma_1,\tilde{\phi}_2,\Gamma_2)&=\int_0^1\frac{1}{2\rho (1-\eta)}\left[\left(\Gamma_{1z}-\rho\eta\right)^2-\rho^2\tilde{\phi}_{1z}^2\right]dz\\
&\quad+\int_0^1\frac{1}{2(\eta+h)}\left[\left(\Gamma_{2z}+\eta\right)^2-\tilde{\phi}_{2z}^2\right]dz-\sqrt{\beta^2-\bar{\omega}^2}+\beta-\frac{\alpha \eta^2}{2},
\end{align*}
where
\begin{equation*}
\bar{\omega}=\omega+\int_0^1\frac{z\tilde{\phi}_{1z}(\Gamma_{1z}-\rho)dz}{\eta-1}+\int_0^1\frac{z\tilde{\phi}_{2z}(\Gamma_{2z}-h)dz}{\eta+h},
\end{equation*}
and Hamilton's equations on $(M\times \mathbb{R}^2,\Omega)$ are
\begin{align}
\dot{\eta}&=\frac{\bar{\omega}}{\sqrt{\beta^2-\bar{\omega}^2}},\label{heq1}\\
\dot{\omega}&=\int_0^1\frac{-\rho}{2(\eta-1)^2}\left[\frac{(\Gamma_{1z}-\rho)^2}{\rho^2}-\tilde{\phi}_{1z}^2\right]+\frac{\rho}{2}dz+\int_0^1\frac{1}{2(\eta+h)^2}\left[(\Gamma_{2z}-h)^2-\tilde{\phi}_{2z}^2\right]-\frac{1}{2}\ dz\nonumber\\
&\quad+\frac{\bar{\omega}}{\sqrt{\beta^2-\bar{\omega}^2}}\left[\int_0^1\frac{z\tilde{\phi}_{1z}(\Gamma_{1z}-\rho)dz}{(\eta-1)^2}+\int_0^1\frac{z\tilde{\phi}_{2z}(\Gamma_{2z}-h)dz}{(\eta+h)^2}\right]+\alpha\eta,\label{heq2}\\
\dot{\tilde{\phi}}_1&=\frac{1}{\eta-1}\left[-\frac{\Gamma_{1z}}{\rho}+\frac{\Gamma_1(1)}{\rho}+\frac{\bar{\omega}}{\sqrt{\beta^2-\bar{\omega}^2}}(z\tilde{\phi}_{1z}-\tilde{\phi}_1(1))\right],\label{heq3}\\
\dot{\Gamma}_1&=\frac{1}{\eta-1}\left[\frac{\bar{\omega}}{\sqrt{\beta^2-\bar{\omega}^2}}z(\Gamma_{1z}-\rho)+\rho\tilde{\phi}_{1z}\right],\label{heq4}\\
\dot{\tilde{\phi}}_2&=\frac{1}{\eta+h}\left[\Gamma_{2z}-\Gamma_2(1)+\frac{\bar{\omega}}{\sqrt{\beta^2-\bar{\omega}^2}}(z\tilde{\phi}_{2z}-\tilde{\phi}_2(1))\right],\label{heq5}\\
\dot{\Gamma}_2&=\frac{1}{\eta+h}\left[\frac{\bar{\omega}}{\sqrt{\beta^2-\bar{\omega}^2}}z(\Gamma_{2z}-h)-\tilde{\phi}_{2z}\right],\label{heq6}\\
\dot{\chi}_1&=\frac{1}{\eta-1}\left(-\Gamma_1(1)+\eta+\frac{\bar{\omega}\tilde{\phi}_1(1)}{\sqrt{\beta^2-\bar{\omega}^2}}\right),\label{heq7}\\
\dot{\chi}_2&=\frac{1}{\eta+h}\left(\Gamma_2(1)+\eta+\frac{\bar{\omega}\tilde{\phi}_2(1)}{\sqrt{\beta^2-\bar{\omega}^2}}\right),\label{heq8}
\end{align}
with boundary conditions
\begin{align}
&\begin{cases}
\rho\tilde{\phi}_{1z}(1)=-\frac{\bar{\omega}(\Gamma_{1z}(1)-\rho)}{\sqrt{\beta^2-\bar{\omega}^2}},\\
\phi_{1z}(0)=0,\end{cases}\label{b12}\\
&\begin{cases}
\tilde{\phi}_{2z}(1)=\frac{\bar{\omega}(\Gamma_{2z}(1)-h)}{\sqrt{\beta^2-\bar{\omega}^2}},\\
\phi_{2z}(0)=0.\end{cases}\label{b22}
\end{align}
We note that the variables $\chi_1$ and $\chi_2$ are cyclic; their conjugate variables $\Gamma_1(1)$, $\Gamma_2(1)$ are therefore conserved and will be set to $0$. Moreover, it is enough to consider \eqref{heq1}--\eqref{heq6}, since $\chi_1$ and $\chi_2$ can be recovered by quadrature from \eqref{heq7} and \eqref{heq8}. The reason for doing this reduction is to eliminate a zero eigenvalue of algebraic multiplicity $2$. Abusing notation, we denote the reduced Hamiltonian system by $(\tilde{M},\mathcal{H},\Omega)$.
Denote by $v_\mathcal{H}$ the  corresponding Hamiltonian vector field, that is the right hand side of \eqref{heq1}--\eqref{heq6}. Then $\mathcal{D}(v_\mathcal{H})=\{ m\in \tilde{N}\cap{Y_1}\ :\text{ such that } \eqref{b12}\text{ and } \eqref{b22} \text{ hold}\}$.
The equilibrium solution is now given by $u=(0,0,0,0,0,0)$, and the linearization $L$ of $v_\mathcal{H}$ around $u$ is
\begin{equation*}
L\left( \begin{array}{c}
\eta \\
\omega \\
\tilde{\phi}_1\\
\Gamma_1\\
\tilde{\phi}_2\\
\Gamma_2 \end{array} \right)
=\left( \begin{array}{c}
\frac{1}{\beta}\left(\omega+\rho\tilde{\phi}_1(1)-\tilde{\phi}_2(1)\right)\\
\left(\alpha-\rho-\frac{1}{h}\right)\eta\\
\frac{\Gamma_{1z}}{\rho}\\
\frac{\rho z}{\beta}\left(\omega+\rho\tilde{\phi}_1(1)-\tilde{\phi}_2(1)\right)-\rho\tilde{\phi}_{1z}\\
\frac{\Gamma_{2z}}{h}\\
-\frac{z}{\beta}\left(\omega+\rho\tilde{\phi}_1(1)-\tilde{\phi}_2(1)\right)-\frac{\tilde{\phi}_{2z}}{h}\end{array} \right)
\end{equation*}
where $\mathcal{D}(L)$ is the set of elements $m\in \tilde{M}\cap Y_1$ which satisfy
\begin{align}
&\begin{cases}
\tilde{\phi}_{1z}(1)=\frac{1}{\beta}(\omega+\rho\tilde{\phi}_{1}(1)-\tilde{\phi}_{2}(1)) ,\\
\tilde{\phi}_{1z}(0)=0,
\end{cases}\label{lb1}\\
&\begin{cases}
\tilde{\phi}_{2z}(1)=-\frac{h}{\beta}(\omega+\rho\tilde{\phi}_{1}(1)-\tilde{\phi}_{2}(1)) ,\\
\tilde{\phi}_{2z}(0)=0.
\end{cases}\label{lb2}
\end{align}

\subsection{Spectrum of $L$}\label{parameterregimes}%kolla namnet
In this section we investigate how the spectrum of $L$ depends on $\alpha$ and $\beta$. In particular we are interested in imaginary eigenvalues, since the dimension of the center manifold is equal to the number of imaginary eigenvalues counted with multiplicity.
\begin{figure}\centering
\includegraphics{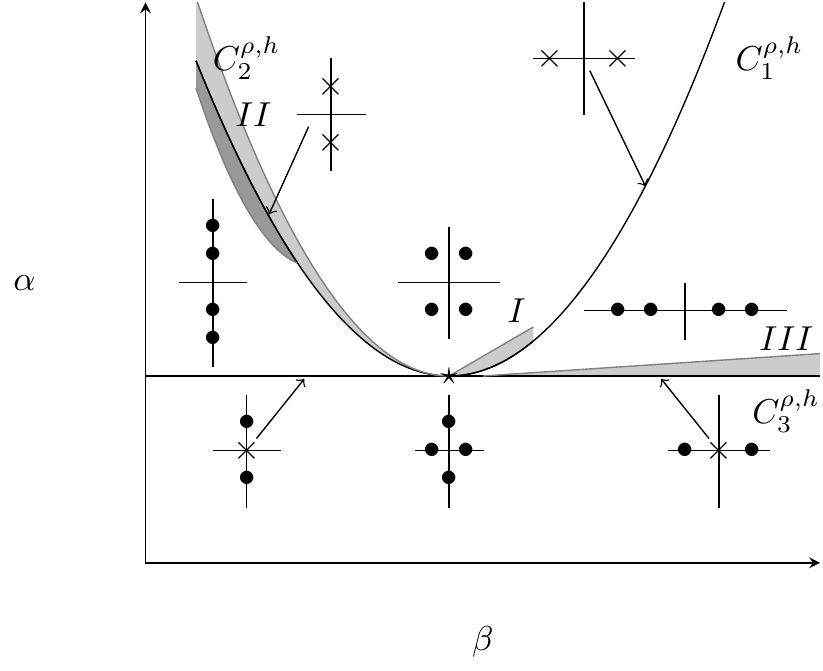}
\caption{Bifurcation curves in the $(\beta,\alpha)$-plane. The $\star$ in the picture indicates the point $(\beta_0,\alpha_0):=(\frac{\rho+h}{3},\rho+\frac{1}{h}$)}
\label{fig2}
\end{figure}
Consider the eigenvalue equation $Lu=\lambda u$, with boundary conditions \eqref{lb1} and \eqref{lb2}. This equation has a solution if and only if 
\begin{equation}
\lambda\left(\frac{\rho}{\tan(\lambda)}+\frac{1}{\tan(h\lambda)}\right)=\alpha-\beta\lambda^2.\label{dispersion}
\end{equation}
When $\lambda=ik$, we obtain the dispersion relation
\begin{equation*}
k\left(\frac{\rho}{\tanh(k)}+\frac{1}{\tanh(hk)}\right)=\alpha+\beta k^2.
\end{equation*}
We will restrict $\lambda$ to some strip $T=\{z\in \mathbb{C}\ : \ \text{Re}(z)\in (-\gamma,\gamma)\}$, where $0<\gamma$ is chosen small enough so that the following is true:
\begin{itemize}
\item When the curve $C_1^{\rho,h}$ is crossed from below, the spectrum of $L$ in $T$ changes from two pairs of real simple eigenvalues to a plus-minus complex-conjugate quartet of complex simple eigenvalues. For points on the curve $C_1^{\rho,h}$ the spectrum is equal to $\{-k_0,k_0\}$ where $\pm k_0$, have algebraic multiplicity $2$.  This change in the spectrum is called a Hamiltonian real 1:1 resonance.
\item When the curve $C_2^{\rho,h}$ is crossed from below, the spectrum of $L$ in $T$ changes from two pairs of simple imaginary eigenvalues to a plus-minus complex-conjugate quartet of complex simple eigenvalues. For points on the curve $C_2^{\rho,h}$ the spectrum is equal to $\{-ik_0,ik_0\}$, where $\pm ik_0$, have algebraic multiplicity $2$. This change in the spectrum is called a Hamiltonian-Hopf bifurcation.
\item When the curve $C_3^{\rho,h}$ is crossed from below, the spectrum of $L$ in $T$ changes from $\{\pm k_1,\pm ik_2\}$, where $\pm k_1$, $\pm ik_2$ are all simple, to two pairs of real simple eigenvalues. For points on the curve $C_3^{\rho,h}$ the spectrum is equal to $\{-k_0,0,k_0\}$, where $\pm k_0$ are simple eigenvalues and $0$ has algebraic multiplicity $2$. This change in the spectrum is called a Hamiltonian real $0^2$ resonance. 
\end{itemize}
The parametrization of the curve $C_1^{\rho,h}$ is given by
\begin{align}
\alpha&=\hat{\alpha}(k_0):=\hat{\beta}(k_0)^2+k_0\left(\frac{\rho}{\tan(k_0)}+\frac{1}{\tan(hk_0)}\right),\label{alphakreal}\\
\beta&=\hat{\beta}(k_0):=\rho\left(\frac{-\sin(k_0)\cos(k_0)+k_0}{2k_0\sin^2(k_0)}\right)+h\left(\frac{-\sin(hk_0)\cos(hk_0)+hk_0}{2hk_0\sin^2(hk_0)}\right)k,\label{betakreal}
\end{align}
and the  parametrization of the curve $C_2^{\rho,h}$ is given by
\begin{align}
\alpha&=\alpha^*(k_0):=-\beta^*(k_0) k_0^2+k_0\left(\frac{\rho}{\tanh(k_0)}+\frac{1}{\tanh(hk_0)}\right)\label{parameterregime2alpha},\\
\beta&=\beta^*(k_0):=\rho\left(\frac{\sinh(k_0)\cosh(k_0)-k_0}{2k_0\sinh^2(k_0)}\right)+h\left(\frac{\sinh(hk_0)\cosh(hk_0)-hk_0}{2hk_0\sinh^2(hk_0)}\right).\label{parameteregime2beta}
\end{align}
\subsection{Change of variables}\label{changeofvariables}
At a later stage, we wish to apply the center manifold theorem. However we cannot apply this theorem directly to the system \eqref{heq1}--\eqref{heq6} since the boundary conditions \eqref{b12}, \eqref{b22} are nonlinear. In order to get linear boundary conditions  we make the change of variables
\begin{align}
&G:\tilde{N}\rightarrow \tilde{M},\nonumber\\
&G(\eta,\omega,\tilde{\phi}_1,\Gamma_1,\tilde{\phi}_2,\Gamma_2)=(\eta,v,\varphi_1,
\Gamma_1,\varphi_2,\Gamma_2)\in \tilde{M},\label{changevar}
\end{align}
where
\begin{align}
&v=\rho\tilde{\phi}_1(1)-\tilde{\phi}_2(1)\label{cvar1},\\
&\varphi_1=\rho\tilde{\phi}_1+W\left(A[\Gamma_1](z)-\frac{\rho}{2}\left(z^2-\frac{1}{3}\right)\right)\label{cvar2},\\
&\varphi_2=\tilde{\phi}_2-W\left(A[\Gamma_2](z)-\frac{h}{2}\left(z^2-\frac{1}{3}\right)\right)\label{cvar3},
\end{align}
and
\begin{align*}
W&=\frac{\bar{\omega}}{\sqrt{\beta^2-\bar{\omega}^2}},\\
 A[f](z)&=\int_0^zsf_s(s)ds-\int_0^1\int_0^zsf_s(s)\ ds \ dz.
\end{align*}
Note that $\varphi_{iz}(1)=\varphi_{iz}(0)=0$ for $i=1,2$ if and only if $\tilde{\phi}_1,\tilde{\phi}_2$ satisfy the boundary conditions \eqref{b12}, \eqref{b22}.
We find that $G$ is invertible in some neighborhood $U_1$ of the origin, with
\begin{equation}
G^{-1}(\eta,v,\varphi_1,\Gamma_1,\varphi_2,\Gamma_2)=\left( \begin{array}{c}
\eta\\
\frac{\beta R}{\sqrt{1+R^2}}-\int_0^1\frac{z\left[\frac{\varphi_{1z}}{\rho}-\frac{Rz(\Gamma_{1z}-\rho)}{\rho}\right](\Gamma_{1z}-\rho)}{\eta-1}\ dz-\int_0^1\frac{z\left[\varphi_{2z}+Rz(\Gamma_{2z}-h)\right](\Gamma_{2z}-h)}{\eta+h}\ dz\\
\frac{\varphi_1}{\rho}-\frac{R}{\rho}\left(A[\Gamma_1](z)-\frac{\rho}{2}(z^2-\frac{1}{3})\right)\\
\Gamma_1\\
\varphi_2+R\left(A[\Gamma_2](z)-\frac{h}{2}(z^2-\frac{1}{3})\right)\\
\Gamma_2 \end{array} \right),\label{ginv}
\end{equation}
where 
\begin{equation*}
R=\frac{\varphi_1(1)-\varphi_2(1)-v}{\left(A[\Gamma_1](1)+A[\Gamma_2](1)-\frac{\rho+h}{3}\right)}.
\end{equation*}
In these new coordinates, the Hamiltonian $H$ is given by
\begin{align}
H(\eta,v,\varphi_1,\Gamma_1,\varphi_2,\Gamma_2)&:=\mathcal{H}(G^{-1}(\eta,v,\varphi_1,\Gamma_1,\varphi_2,\Gamma_2))\nonumber\\
&=\int_0^1\frac{1}{2\rho(1-\eta)}\left((\Gamma_{1z}-\rho\eta)^2-\left[\varphi_{1z}-Rz(\Gamma_{1z}-\rho)\right]^2\right)\ dz\nonumber\\
&\quad +\int_0^1\frac{1}{2(\eta+h)}\left((\Gamma_{2z}+\eta)^2-\left[\varphi_{2z}+Rz(\Gamma_{2z}-h)\right]^2\right)\ dz\nonumber\\
&\quad -\frac{\beta}{\sqrt{1+R^2}}+\beta-\frac{\alpha\eta^2}{2},\label{hamiltonianm}
\end{align}
and equations \eqref{heq1}--\eqref{heq6} become
\begin{align}
\dot{\eta}&=R\label{meq1}, \\
\dot{v}&=\frac{1}{\eta-1}\left(-\Gamma_{1z}(1)+R\left[-R(\Gamma_{1z}(1)-\rho)-\varphi_1(1)+R(A[\Gamma_1](1)-\frac{\rho}{3})\right]\right)\nonumber\\
&\quad -\frac{1}{\eta+h}\left(\Gamma_{2z}(1)+R\left[R(\Gamma_{2z}(1)-h)-\varphi_2(1)-R(A[\Gamma_2](1)-\frac{h}{3})\right]\right),\label{meq2}\\
\dot{\varphi}_1&=\frac{1}{\eta-1}\left(-\Gamma_{1z}+R\left[z\varphi_{1z}-Rz^2(\Gamma_{1z}-\rho)-\varphi_1(1)+R(A[\Gamma_1](1)-\frac{\rho}{3})\right]\right)\nonumber\\
&\quad+\frac{\dot{\bar{\omega}}(1+R^2)^\frac{3}{2}}{\beta}\left(A[\Gamma_1](z)-\frac{\rho}{2}\left(z^2-\frac{1}{3}\right)\right)+\frac{RA[\varphi_{1z}](z)}{\eta-1},\label{meq3}\\
\dot{\Gamma}_1&=\frac{\varphi_{1z}}{\eta-1},\label{meq4}\\
\dot{\varphi}_2&=\frac{1}{\eta+h}\left(\Gamma_{2z}+R\left[z\varphi_{2z}+Rz^2(\Gamma_{2z}-h)-\varphi_2(1)-R(A[\Gamma_2][1]-\frac{h}{3})\right]\right)\nonumber\\
&\quad -\frac{\dot{\bar{\omega}}(1+R^2)^\frac{3}{2}}{\beta}\left(A[\Gamma_2](z)-\frac{h}{2}\left(z^2-\frac{1}{3}\right)\right)+\frac{RA[\varphi_{2z}](z)}{\eta+h},\label{meq5}\\
\dot{\Gamma}_2&=-\frac{\varphi_{2z}}{\eta+h},\label{meq6}
\end{align}
where 
\begin{equation}
\dot{\bar{\omega}}=-\frac{(1+R^2)(\Gamma_{1z}(1)-\rho)^2}{2\rho(\eta-1)^2}+\frac{(1+R^2)(\Gamma_{2z}(1)-h)^2}{2(\eta+h)^2}+\frac{\rho}{2}-\frac{1}{2}+\alpha\eta\label{omegabardot}.
\end{equation}
In conclusion, the Hamiltonian vector field $v_H$ corresponding to the Hamiltonian $H$, is given by the right hand side of equations \eqref{meq1}--\eqref{meq6}. 
The linearisation $K$ of $v_H$  around the trivial solution $u=(0,0,0,0,0,0)$ is given by
\begin{equation*}
K\left( \begin{array}{c}
\eta \\
v \\
\varphi_1\\
\Gamma_1\\
\varphi_2\\
\Gamma_2 \end{array} \right)
=\left( \begin{array}{c}
-\frac{3(\varphi_1(1)-\varphi_2(1)-v)}{\rho+h}\\
\Gamma_{1z}(1)-\frac{1}{h}\Gamma_{2z}(1) \\
\Gamma_{1z}+\left(\Gamma_{1z}(1)-\frac{\Gamma_{2z}(1)}{h}+(\alpha-\rho-\frac{1}{h})\eta\right)\frac{\rho(\frac{1}{3}-z^2)}{2\beta}\\
-\varphi_{1z}\\
\frac{\Gamma_{2z}}{h}-\left(\Gamma_{1z}(1)-\frac{\Gamma_{2z}(1)}{h}+(\alpha-\rho-\frac{1}{h})\eta\right)\frac{h(\frac{1}{3}-z^2)}{2\beta}\\
-\frac{\varphi_{2z}}{h} \end{array} \right),
\end{equation*}
with 
\begin{equation*}
\mathcal{D}(K)=\{(\eta,v,\varphi_1,\Gamma_1,\varphi_1,\Gamma_2)\in \tilde{M} \ :\ \varphi_i\in H^2(0,1),\ \Gamma_i\in H_0^2(0,1),\ \varphi_{iz}(0)=\varphi_{iz}(1)=0\}.
\end{equation*}
The operator $K$ can also be determined by using that 
\begin{equation}
K=dG[0]LdG[0]^{-1}.\label{linop}
\end{equation}
This is well defined since $dG[0]^{-1}:\ \mathcal{D}(K)\rightarrow \mathcal{D}(L)$. Due to this we may work with $L$ instead of $K$ when doing spectral analysis.

Finally we note that $v_H$ anticommutes with the symmetry
\begin{equation}
S:(\eta,v,\varphi_1,\Gamma_1,\varphi_2,\Gamma_2)\rightarrow (\eta,-v,-\varphi_1,\Gamma_1,-\varphi_2,\Gamma_2)\label{reverser},
\end{equation}
that is, the system \eqref{meq1}-\eqref{meq6} is reversible with reverser $S$. Observe also that $H^\mu$ is invariant under $S$.
\section{Center manifold reduction}\label{cmtsection}
In this section we will prove the existence of solitary wave solutions of Hamilton's equations for certain parameter regimes.  In order to do this we show that the reduced Hamiltonian system obtained when applying the center manifold theorem, admits homoclinic solutions. Before stating this theorem we introduce the bifurcation parameter $\mu$ by writing $(\beta,\alpha)=(\tilde{\beta},\tilde{\alpha})+\mu$, where $\mu=(\mu_1,\mu_2)$. The Hamiltonian vector field $v_H$ is then denoted by $v_H^\mu$ and in the expression for $K$ we write $\tilde{\alpha}$ and $\tilde{\beta}$ instead of $\alpha$ and $\beta$, so that $K$ is the linearization of $v_H^0$ around $(0,0,0,0,0,0)$. In the same way we will write $H^\mu$ and $G^\mu$ instead of $H$ and $G$. We will use the following version of the center manifold theorem which is due to \cite{MI} and was used in for example \cite{BG}.
%skriver att lambda är ett öppet område kring 0, ska vara mu0
\begin{theorem}\label{cmt}
Consider the differential equation
\begin{equation}
\dot{u}=Ku+F(u,\mu),\label{fulleqthm}
\end{equation}
where $u$ belongs to a Hilbert space $E$, $\mu\in \mathbb{R}^n$ is a parameter and $K:\mathcal{D}(K)\subset E\rightarrow E$ is a closed linear operator. Suppose that \eqref{fulleqthm} is Hamilton's equations for the Hamiltonian system $(E,\Omega,\mathcal{H})$. Also suppose that $0$ is an equilibrium point of \eqref{fulleqthm} and that
\begin{itemize}
\item[H1] $E$ has two closed, $K$-invariant subspaces $E_1$, $E_2$ such that
\begin{align*}
E&=E_1\oplus E_2,\\
\dot{u}_1&=K_1u_1+F_1(u_1+u_2,\mu),\\
\dot{u}_2&=K_2u_2+F_2(u_1+u_2,\mu),
\end{align*}
where $K_i=K\vert_{\mathcal{D}(K)\cap E_i}:\mathcal{D}(K)\cap E_i\rightarrow E_i$, $i=1,2$ and $F_1=PF$, $F_2=(I-P)F$, where $P$ is the projection of $E$ onto $E_1$.
\item[H2] $E_1$ is finite dimensional and the spectrum of $K_1$ lies on the imaginary axis.
\item[H3] The imaginary axis lies in the resolvent set of $K_2$ and 
\begin{equation*}
\norm{(K_2-iaI)^{-1}}\leq \frac{C}{1+\abs{a}},\quad a\in \mathbb{R}.
\end{equation*}
\item[H4]
There exists $k\in \mathbb{N}$ and neighborhoods $\Lambda\subset \mathbb{R}^n$ and $U\subset \mathcal{D}(K)$ of $0$ such that $F$ is $k+1$ times continuously differentiable on $U\times \Lambda$ and the derivatives of $F$ are bounded and uniformly continuous on $U\times \Lambda$ with
\begin{equation*}
F(0,\mu_0)=0,\quad d_1F[0,\mu_0]=0.
\end{equation*}
\end{itemize}
Under the hypothesis $H1-H4$ there exist neighborhoods $\tilde{\Lambda}\subset \Lambda$ and $\tilde{U}_1\subset U\cap E_1$, $\tilde{U}_2\subset U\cap E_2$ of zero and a reduction function $r:\tilde{U}_1\times \tilde{\Lambda}\rightarrow \tilde{U}_2$ with the following properties. The reduction function $r$ is $k$ times continuously differentiable on $\tilde{U}_1\times \tilde{\Lambda}$ and the derivatives of $r$ are bounded and uniformly continuous on $\tilde{U_1}\times \tilde{\Lambda}$ with
\begin{equation*}
r(0,\mu_0)=0,\quad d_1r[0,\mu_0]=0.
\end{equation*}
The graph
\begin{equation*}
X_C^{\mu}=\{u_1+r(u_1,\mu)\in \tilde{U}_1\times \tilde{U}_2\ :\ u_1\in \tilde{U}_1\},
\end{equation*}
is a Hamiltonian center manifold for \eqref{fulleqthm} with the following properties:
\begin{itemize}
\item Through every point in $X_C^\mu$ there passes a unique solution of \eqref{fulleqthm} that remains on $X_C^\mu$ as long as it remains in $\tilde{U}_1\times \tilde{U}_2$. We say that $X_C^\mu$ is a locally invariant manifold of \eqref{fulleqthm}.
\item Every small bounded solution $u(x)$, $x\in \mathbb{R}$ of \eqref{fulleqthm} that satisfies $u_1(x)\in \tilde{U}_1$, $u_2(x)\in \tilde{U}_2$ lies completely in $X_C^\mu$.
\item Every solution $u_1$ of the reduced equation 
\begin{equation}
\dot{u}_1=K_1u_1+F_1(u_1+r(u_1,\mu),\mu)\label{thmredeq},
\end{equation}
generates a solution
\begin{equation*}
u(x)=u_1(x)+r(u_1(x),\mu)
\end{equation*}
of \eqref{fulleqthm}.
\item $X_C^\mu$ is a symplectic submanifold of $E$, and the flow determined by the Hamiltonian system $(X_C^\mu,\tilde{\Omega},\tilde{H})$, where the tilde denotes restriction to $X_C^\mu$, coincides with the flow on $X_C^\mu$ determined by $(E,\Omega,H)$. The reduced equation \eqref{thmredeq} represents Hamilton's equations for $(X_C^\mu,\tilde{\Omega},\tilde{H})$.
\item 
If \eqref{fulleqthm} is reversible, that is if there exists a linear symmetry $S$ which anticommutes with the right hand side of \eqref{fulleqthm}, then the reduction function $r$ can be chosen so that it commutes with $S$.
\end{itemize}
\end{theorem}
In our case we have $E=\tilde{M}$ and \eqref{fulleqthm} correspond to
\begin{equation}
\dot{v}=Kv+F(v,\mu)\label{meqfull},
\end{equation}
where $F(v,\mu)=v_H^\mu(v)-Kv$. The bifurcation parameter $\mu$ belongs to some neighborhood $\Lambda$ of $0$ in $\mathbb{R}^2$, which will be specified later on. We will use the same arguments as in \cite{BGT} when showing that hypothesis $H1-H4$ are satisfied.
Note that $H3$ is satisfied, by the following theorem:
\begin{theorem}
There exist constants $C,a_0>0$ such that 
\begin{equation}
\norm{(L-iaI)^{-1}}\leq \frac{C}{\abs{a}}\label{resineq}
\end{equation}
for all $\abs{a}> a_0$
\end{theorem}
The proof of this theorem is very similar to the proof of proposition 3.2 in \cite{BGT} and will therefore be omitted. It follows from \eqref{linop} that \eqref{resineq} holds for $K$ as well. In particular we get from \eqref{resineq} that the resolvent set of $K$ is nonempty. This implies that $K$ is closed.

Let $ia$ be an element in the resolvent set of $K$. It follows from the Kondrachov embedding theorem that 
\begin{equation*}
K:\mathcal{D}(K)\rightarrow \tilde{M},
\end{equation*}
has compact resolvent. This implies that the spectrum of $K$ consists of an at most countable number of isolated eigenvalues with finite multiplicity. From this we conclude that there exists $\xi>0$ such that
\begin{equation*}
\text{Spec}(K)=\{\lambda\in\text{Spec}(K)\ : \ \abs{\text{Re}(\lambda)}>\xi\}\cup \{\lambda \in \text{Spec}(K)\ :\ \text{Re}(\lambda)=0\},
\end{equation*}
that is, the part of the spectrum which lies on the imaginary axis is separated from the rest of the spectrum. This allows us to define the spectral projection $P$, corresponding to the imaginary part of the spectrum:
\begin{equation}
P=-\frac{1}{2\pi i}\int_\gamma (K-\lambda)^{-1}d\lambda, \label{specprojection}
\end{equation}
where $\gamma$ is a curve surrounding the imaginary part of the spectrum and which lies in the resolvent set.

%The Hamiltonian $H$ from \eqref{hamiltonianm} takes the form
%\begin{align*}
%H^\mu(\eta,v,\varphi_1,\Gamma_1,\varphi_2,\Gamma_2)&=\int_0^1\frac{1}{2\rho(1-\eta)}\left((\Gamma_{1z}-\rho\eta)^2-\left[\varphi_{1z}-Rz(\Gamma_{1z}-\rho)\right]^2\right)dz\nonumber\\
%&\quad +\int_0^1\frac{1}{2(\eta+h)}\left[(\Gamma_{2z}+\eta)^2-\left[\varphi_{2z}+Rz(\Gamma_{2z}-h)\right]^2\right)dz\nonumber\\
%&\quad -\frac{\tilde{\beta}+\mu_1}{\sqrt{1+R^2}}+\tilde{\beta}+\mu_1-\frac{(\tilde{\alpha}+\mu_2)\eta^2}{2}. 
%\end{align*}
We check hypotheses $H1,H2$ and $H4$ of theorem \ref{cmt}. Recall from section \ref{changeofvariables} that there exists a neighborhood $U_1\subset\tilde{N}$ such that $G^\mu$ is invertible in that neighborhood, for a fixed value of $\mu$. For a fixed $\tilde{\beta}$ we may assume that $\abs{\bar{\omega}}<\tilde{\beta}/2$ for all $u\in U_1$. Choose a neighborhood $\Lambda\subset \mathbb{R}^2$ of $(0,0)$ such that $\abs{\mu_1}<\tilde{\beta}/4$, $ \tilde{\alpha}+\mu_2>0$, for all $\mu\in \Lambda$. It follows that $G^\mu$ is invertible in the neighborhood $U_1\times \Lambda$ of the origin. Let $U_2=G(U_1)$. In $H4$  we let $U=U_2\cap\mathcal{D}(K)$  It is easy to see that $F$ is smooth in $U\times \Lambda$.
We also define a symplectic product $\Omega_2^\mu$ in $U_2$ by
\begin{equation*}
(\Omega_2^\mu(u,u^*))_m:=\Omega(dG^\mu[(G^\mu)^{-1}(m)]^{-1}(u),dG^\mu[(G^\mu)^{-1}(m)]^{-1}(u^*)).
\end{equation*}
Let $E_1=P(\tilde{M})$ and let $E_2=(I-P)\tilde{M}$, where $P:\tilde{M}\rightarrow \tilde{M}$ is the spectral projection corresponding to the imaginary part of the spectrum, defined as in \eqref{specprojection}. It follows from theorem 6.17 chapter III in \cite{KA}, together with the fact that the imaginary part of the spectrum of $K$ consists of a finite number of eigenvalues with finite multiplicity (see section \ref{parameterregimes}), that $H1$ and $H2$ are satisfied.
By the center manifold theorem, there exists neighborhoods $\tilde{U}_1\subset U \cap E_1$, $\tilde{\Lambda}\subset\Lambda$ of zero and a reduction function $r:\tilde{U}_1\times \tilde{\Lambda}\rightarrow\tilde{U}_2$ such that $r(0,0)=d_1r[0,0]=0$ and
\begin{equation*}
X_C^\mu=\{u_1+r(u_1,\mu)\ : \ u_1\in \tilde{U}_1\},
\end{equation*}
is a center manifold for \eqref{meqfull} . We then have the Hamiltonian system $(X_C^\mu,\tilde{\Omega}_2^\mu, \tilde{H}^\mu)$, where 
\begin{align*}
(\tilde{\Omega}_2^\mu)_m(u_1,u_1^*)&=(\Omega_2^\mu)_{m+r(m,\mu)}(u_1+d_1r[m,\mu](u_1),u_1^*+d_1r[m,\mu](u_1^*),\\
\tilde{H}^\mu(u_1)&=H^\mu(u_1+r(u_1,\mu)).
\end{align*}
The Hamiltonian $\tilde{H}^\mu$ is called the reduced Hamiltonian. Recall that $v_H^\mu$ is reversible with reverser $S$ given in \eqref{reverser}. By theorem \ref{cmt} the reduction function $r$ can be chosen so that it commutes with $S$. If we combine this with the fact that $H^\mu$ is invariant under $S$, we get that
\begin{align*}
 \tilde{H}^\mu(Su_1)&=H^\mu(Su_1+r(Su_1,\mu))\\
&=H^\mu(S(u_1+r(u_1,\mu)))\\
&=H^\mu(u_1+r(u_1,\mu)\\
&=\tilde{H}^\mu(u_1),
\end{align*}
that is, $\tilde{H}^\mu$ is invariant under $S$.

From Darboux's theorem (see  theorem 4 of \cite{BG}) there exists a near identity change of variables
\begin{equation*}
\hat{u}_1=u_1+\theta(u_1,\mu),
\end{equation*}
such that $\tilde{\Omega}_2^\mu$ is transformed into $\Psi$, where 
\begin{equation*}
\Psi(u_1,u_1^*)=(\Omega_2^0)_0(u_1,u_1^*)
\end{equation*}
The coordinate map is then given by $\hat{u}_1+\hat{r}(u_1,\mu)$, where $\hat{r}: \tilde{U}_1\times\tilde{\Lambda} \rightarrow \tilde{U}_1\times \tilde{U}_2$ and $\hat{r}(0,0)=d_1\hat{r}(0,0)=0$.
For simplicity we drop the $\hat{}$ notation. It follows from the definition of the Hamiltonian vector field, that $2H_2^0[u,u^*]=\Psi(Ku,u^*)$. We conclude that $K$ is skew-symmetric with respect to $\Psi$, since
\begin{equation} 
\Psi(Ku,u^*)=2H_2^0[u,u^*]=2H_2^0[u^*,u]=\Psi(Ku^*,u)=-\Psi(u,Ku^*)\label{skewsymmetry}
\end{equation}
\subsection{Hamiltonian-Hopf bifurcation}\label{hamhopf}
Recall from section \ref{parameterregimes} that a Hamiltonian-Hopf bifurcation occurs when crossing the curve $C_2^{\rho,h}$. Let $(\beta,\alpha)=(\beta^*(k),\alpha^*(k))+\mu$, where $\alpha^*(k)$ and $\beta^*(k)$  are given in \eqref{parameterregime2alpha}, \eqref{parameteregime2beta} and $\mu=(0,\delta)$, that is we consider region $II$ in figure \ref{fig2}. We will show that there exists $k,\rho$ and $h$ such that the reduced Hamiltonian system has homoclinic solutions for sufficiently small $\delta$.
We saw in section \ref{parameterregimes} that when $\alpha=\alpha^*(k)$ and $\beta=\beta^*(k)$, the spectrum of $L$ in $T$ is equal to $\{\pm ik\}$, where $\pm ik$ are eigenvalues of algebraic multiplicity 2. Generalized eigenvectors corresponding to the eigenvalue $ik$ are given by
\begin{equation}
e_1=\left( \begin{array}{c}
1\\
\frac{1}{ik}(\alpha-\rho-\frac{1}{h})\\
i\left(\frac{\cosh(kz)}{\sinh(k)}-\frac{1}{k}\right)\\
\rho\left(-\frac{\sinh(kz)}{\sinh(k)}+z\right)\\
i\left(-\frac{\cosh(hkz)}{\sinh(hk)}+\frac{1}{hk}\right)\\
\frac{\sinh(hkz)}{\sinh(hk)}-z
\end{array}\right)
,\quad 
e_2=\left(\begin{array}{c}
0\\
\frac{1}{k^2}(\alpha-\rho-\frac{1}{h})\\
\frac{\cosh(kz)}{\sinh(k)}\left(\scalebox{0.8}{$\tanh(kz)z$}-\frac{1}{\tanh(k)}\right)+\frac{1}{k^2}\\
-\frac{i\rho\cosh(kz)}{\sinh(k)}\left(\frac{\tanh(kz)}{\tanh(k)}-z\right)\\
-\frac{h\cosh(hkz)}{\sinh(hk)}\left(\scalebox{0.8}{$\tanh(hkz)z$}-\frac{1}{\tanh(hk)}\right)-\frac{1}{hk^2}\\
\frac{ih\cosh(hkz)}{\sinh(hk)}\left(\frac{\tanh(hkz)}{\tanh(hk)}-z\right)\end{array}\right)\label{geneigenik}
\end{equation}
and satisfy $(L-ik)e_1=0$ and $(L-ik)e_2=e_1$. Since $L$ is real, it follows that $\bar{e}_1$ and $\bar{e}_2$ are generalized eigenvectors corresponding to the eigenvalue $-ik$.
Note that
\begin{equation*}
\Omega(e_1,\bar{e}_2)=\gamma_1:=-\beta^*(k)+\frac{1}{k}\left(\frac{\rho}{\tanh(k)}+\frac{1}{\tanh(hk)}\right)-\frac{\rho k}{\tanh(k)\sinh^2(k)}-\frac{h^2k}{\tanh(hk)\sinh^2(hk)},
\end{equation*}  
and $\gamma_1>0$.
Moreover, $\Omega(e_1,\bar{e}_1)=0$ and $\Omega(e_2,\bar{e}_2)=-i\gamma_2$, where $\gamma_2>0$. Let $e_2'=e_2+\frac{i\gamma_2}{2\gamma_1}e_1$. Then 
\begin{align*}
\Omega(e_1,e_2')&=0,\ \Omega(e_1,\bar{e}_2')=\gamma_1,\ \Omega(\bar{e}_1,e_2')=\gamma_1,\  \Omega(e_2',\bar{e}_2')=0.
\end{align*}
We therefore let
\begin{equation}
W_1=\frac{e_1}{\sqrt{\gamma_1}},\quad W_2=\frac{e_2'}{\sqrt{\gamma_1}}.\label{symplecticik}
\end{equation}
The vectors $W_i$ satisfies
\begin{equation*}
\Omega(W_1,\bar{W}_2)=\Omega(\bar{W}_1,W_2)=1,
\end{equation*}
and every other combination is equal to $0$.
In conclusion, $\{W_1,\bar{W}_2,\bar{W}_1,W_2\}$ is a symplectic basis for the generalized eigenspace corresponding to $\pm ik$.

Let $f_i=dG(0)(W_i)$, where $W_i$ are given in \eqref{symplecticik}. Note that
\begin{align*}
\Psi(f_i,f_j)&=(\Omega_2^0)_0(f_i,f_j)\\
&=\Omega(dG^\mu[(G^\mu)^{-1}(0)]^{-1}(f_i),dG^\mu[(G^\mu)^{-1}(0)]^{-1}(f_j))\\
&=\Omega(dG^\mu(0)^{-1}(f_i),dG^\mu(0)^{-1}(f_j))\\
&=\Omega(W_i,W_j).
\end{align*}
It follows that $\{f_1,f_2,\bar{f}_1,\bar{f}_2\}$ is a symplectic basis of $(E_1,\Psi)$, since $\{W_1,W_2,\bar{W}_1,\bar{W}_2\}$ is a symplectic basis of the generalized eigenspace corresponding to $\pm ik$, with respect to $\Omega$.
We write
\begin{equation*}
u_1=(A,B,\bar{A},\bar{B})=Af_1+Bf_2+\bar{A}\bar{f}_1+\bar{B}\bar{f}_2,
\end{equation*}  
and see that
\begin{equation*}
\Psi((A,B,\bar{A},\bar{B}),(A^*,B^*,\bar{A}^*,\bar{B}^*))=A\bar{B}^*-A^*\bar{B}+\bar{A}B^*-\bar{A}^*B.
\end{equation*}
Moreover, the action of the reverser $S$ is given in coordinates by
\begin{equation*}
S\:\ (A,B,\bar{A},\bar{B})\rightarrow (\bar{A},-\bar{B},A,-B).
\end{equation*}
As in \cite{BG}, we use normal form theory for Hamiltonian systems (see \cite{MH}) to conclude that for every $n_0\geq 2$, there exists a near identity, analytic, symplectic change of variables, such that
\begin{equation}
\tilde{H}^\mu(A,B)=ik(A\bar{B}-\bar{A}B)+\abs{B}^2+H_{NF}(\abs{A}^2,i(A\bar{B}-\bar{A}B),\delta)+\mathcal{O}(\abs{u_1}^2\abs{(\delta,u_1)}^{n_0}),\label{normalformreduction}
\end{equation}
where $H_{NF}$ is a real polynomial of degree $n_0+1$ which satisfies
\begin{equation*}H_{NF}(\abs{A}^2,i(A\bar{B}-\bar{A}B),\delta)=\mathcal{O}(\abs{u_1}\abs{(\delta,u_1)}).
\end{equation*}
Take $n_0=3$. Then
\begin{align}
H_{NF}(\abs{A}^2,i(A\bar{B}-\bar{A}B),\delta)&=c_1\delta\abs{A}^2+ic_2\delta(A\bar{B}-\bar{A}B)+c_3\abs{A}^4+ic_4\abs{A}^2(A\bar{B}-\bar{A}B)\\
&\quad -c_5(A\bar{B}-\bar{A}B)^2+\delta^2c_6\abs{A}^2+ic_7\delta^2(A\bar{B}-\bar{A}B).\label{normalform}
\end{align}
We have that
\begin{align}
\dot{u}&=Ku+F(u,\mu)\label{hamfull},\\
\dot{u}_1&=Ku_1+N(u_1,\mu),\label{hamreduced}
\end{align}
where $u=u_1+r(u_1,\mu)\in X_C^\mu$ and where $N$ is the nonlinear part of the reduced Hamiltonian vector field. The identity \eqref{hamfull} is just Hamilton's equation and \eqref{hamreduced} is obtained when we in \eqref{hamfull} project onto $E_1$. This means in particular that $N(u_1,\mu)=PF(u_1+r(u_1,\mu),\mu)$. By combining \eqref{hamreduced} with \eqref{hamfull} we get that
\begin{align}
&\dot{u}=Ku+F(u,\mu)\Leftrightarrow\nonumber\\
&\dot{u}_1+d_1r[u_1,\mu](\dot{u}_1)=K(u_1+r(u_1,\mu))+F(u_1+r(u_1,\mu),\mu)\Leftrightarrow\nonumber\\
&K(r(u_1,\mu))-d_1r[u_1,\mu](Ku_1)=-F(u_1+r(u_1,\mu),\mu)+N(u_1,\mu)+d_1r[u_1,\mu](N(u_1,\mu))\label{hamfullmod}.
\end{align}
Since $F(0,\mu)=0$, for all $\mu\in\tilde{\Lambda}$, we may assume as in lemma 4.4 of \cite{BGT} that $r(0,\mu)=0$ for all $\mu\in \tilde{\Lambda}$. The Taylor expansions of $r$ and $N$ are then given by
\begin{align*}
N(u_1,\mu)&=\sum_{h+i+j+k+l=2}^3N_{ijkl}^h\delta^hA^iB^j\bar{A}^k\bar{B}^l+\mathcal{O}(\abs{u_1}\abs{(\delta,u_1)}^3),\\
r(u_1,\mu)&=\sum_{h+i+j+k+l=2}^3r_{ijkl}^h\delta^hA^iB^j\bar{A}^k\bar{B}^l+\mathcal{O}(\abs{u_1}\abs{(\delta,u_1)}^3).
\end{align*}
By inserting \eqref{normalform} into \eqref{normalformreduction} we obtain the reduced Hamiltonian equations.
\begin{align}
A_x&=ikA+B+iA(c_2\delta+c_4\abs{A}^2)+2c_5A(-A\bar{B}+\bar{A}B)\nonumber\\
&\quad +ic_7\delta^2A+\mathcal{O}(\abs{u_1}\abs{(\delta,u_1}^{3})\label{redhameq1},\\
B_x&=ikB+iB(c_2\delta+c_4\abs{A}^2+c_7\delta^2)-A(\delta c_1+2c_3\abs{A}^2+c_4iA\bar{B}+c_6\delta^2)\nonumber\\
&\quad +2c_5B(-A\bar{B}+\bar{A}B)+\mathcal{O}(\abs{u_1}\abs{(\delta,u_1)}^{3})\label{redhameq2},
\end{align}
and $N(u_1,\mu)$ is identified as the nonlinear part of \eqref{redhameq1} and \eqref{redhameq2}. 
The following theorem shows that the signs of $c_1$ and $c_3$ determine which type of solitary wave solutions we have for \eqref{redhameq1}, \eqref{redhameq2}.
\begin{theorem}\label{caseshamhopf}
Suppose that $c_1<0$.
\begin{enumerate}
\item\cite{IP} $c_3>0$: For each sufficiently small positive value of $\delta$ the system \eqref{redhameq1}, \eqref{redhameq2} has two distinct symmetric homoclinic solutions.
\item \cite{BG} $c_3>0$: For each sufficiently small, positive value of $\delta$ the system \eqref{redhameq1}, \eqref{redhameq2} has two one-parameter families of geometrically distinct homoclinic solutions which generically resemble multiple copies of one of the homoclinic solutions in (1).
\item \cite{IP} $c_3<0$: For each sufficiently small, negative value of $\delta$  the system \eqref{redhameq1}, \eqref{redhameq2} has a one-parameter family of pairs of reversible homoclinic orbits to periodic orbits.
\end{enumerate}
The homoclinic solutions in $1$ and $2$, correspond to envelope solitary waves of amplitude $\mathcal{O}(-c_1\delta^{\frac{1}{2}})$ which decay to a horizontal flow as $x\rightarrow \pm \infty$. These waves are sometimes called bright solitary waves. The solutions found in $3$ are called dark solitary waves. The amplitude of such a wave around $0$ is small in comparison with the amplitude as $\abs{x}\rightarrow \infty$ (see figure \ref{darkbright} for typical wave profiles). 
\end{theorem}
The terms bright and dark solitary waves come from the nonlinear Schrödinger equation. The connection between the system \eqref{redhameq1}--\eqref{redhameq2} and the nonlinear Schrödinger equation can be seen by scaling:
\begin{equation*}
\tilde{x}=\abs{\delta}^\frac{1}{2}x,\quad\abs{\delta}^\frac{1}{2}\tilde{A}(\tilde{x})=A(x)\exp(-ikx),\quad \abs{\delta} \tilde{B}(\tilde{x})=B(x)\exp(-ikx).
\end{equation*}
When terms of $\mathcal{O}(\abs{\delta}^\frac{1}{2})$ are neglected, we find that $\tilde{A}$ satisfies
\begin{equation*}
\tilde{A}_{\tilde{x}\tilde{x}}+\sgn(\delta)c_1\tilde{A}+2c_3\tilde{A}\abs{\tilde{A}}^2=0.
\end{equation*}
The case $c_3>0$ corresponds to the focusing NLS equation and the case $c_3<0$ corresponds to the defocusing NLS equation.

First we find the coefficient $c_1$. We do this by considering terms of order $\delta A$ in \eqref{hamfullmod}:
\begin{align*}
Kr_{1000}^1- ikr_{1000}^1&=-F_{1}^1(f_1)+ic_2f_1-c_1f_2\Leftrightarrow\\
&(K-ikI)r_{1000}^1=-F_{1}^1(f_1)+i\delta c_2f_1-c_1f_2\Rightarrow\\
&\Psi((K-ikI)r_{1000}^1,\bar{f}_{1})=-\Psi(F_{1}^1(f_1),\bar{f}_1)+c_1
\end{align*}
where $F_{1}^1(u_1)=d_2d_1F[0,0](u_1)$. Using that $K$ is skew symmetric with respect to $\Psi$ and the fact that $(K+ikI)\bar{f}_{1}=0$, we find that $\Psi((K-ikI)r_{1000}^1,\bar{f}_{1})=-\Psi(r_{1000}^1,(K+ikI)\bar{f}_{1})=0$. It follows that
\begin{equation}
c_1=\Psi(F_{1}^1(f_1),\bar{f}_1).\label{coeffc1}
\end{equation}
Next we find an expression for $c_3$. In order to do this we consider terms of order $\abs{u_1}^2$ and $\abs{u_1}^3$ in \eqref{hamfullmod}. This gives us the following two equations:
\begin{align}
\abs{u_1}^2:\quad &Kr_{2}^0(u_1)-d_1r_{2}^0[u_1](Ku_1)=-F_{2}^0(u_1,u_1)+N_{2}^0(u_1)\label{secondorderham},\\
\abs{u_1}^3:\quad &Kr_{3}^0(u_1)-d_1r_{3}^0[u_1](Ku_1)=-F_{3}^0(u_1,u_1,u_1)-2F_{2}^0(u_1,r_{2}^0(u_1))+d_1r_2^0[u_1](N_{2}^0(u_1))\nonumber\\
&+N_{3}^0(u_0),\label{thirdorderham}
\end{align}
where
\begin{align*}
r_{n}^m(u_1,\delta)&=\sum_{i+j+k+l=n}r_{ijkl}^m\delta^mA^iB^j\bar{A}^k\bar{B}^l,\\
N_{n}^m(u_1,\delta)&=\sum_{i+j+k+l=n}N_{ijkl}^m\delta^mA^iB^j\bar{A}^k\bar{B}^l,\\
F_{2}^0(u,u^*)&=\frac{d_1^2F[0,0](u,u^*)}{2},\\
F_{3}^0(u,u^*,\hat{u})&=\frac{d_1^3F[0,0](u,u^*,\hat{u})}{6},\\
\end{align*}
Equating coefficients of the terms order $A^2$ and $\abs{A}^2$ in \eqref{secondorderham}, we obtain
\begin{align}
\abs{A}^2&:\quad Kr_{1010}^0=-2F_{2}^0(f_1,\bar{f}_1)\label{r01010},\\
A^2&:\quad (K-2ik)r_{2000}^0=-F_{2}^0(f_1,\bar{f}_1)\label{r02000}.
\end{align}
The vectors $r_{1010}^0$ and $r_{2000}^0$ can be calculated by solving the equations \eqref{r01010} and \eqref{r02000}.
In the same way, considering $A\abs{A}^2$ in \eqref{thirdorderham}, we obtain
\begin{align}
A\abs{A}^2:\quad (K-ik)r_{2010}^0&=-3F_{3}^0(f_1,f_1,\bar{f_1})-2F_{2}^0(f_1,r_{1010}^0)-2F_{2}^0(\bar{f}_1,r_{2000}^0)+c_4f_1-2c_3f_2\Rightarrow\nonumber\\
c_3&=\Psi(F_{2}^0(f_1,r_{1010}^0),\bar{f}_1)+\Psi(F_{2}^0(\bar{f}_1,r_{2000}^0),\bar{f}_1)+\frac{3}{2}\Psi(F_{3}^0(f_1,f_1,\bar{f}_1),\bar{f_1})\label{coeffc3},
\end{align}
where we again used that $\Psi((K-ik)r_{2010}^0,\bar{f}_1)=0$, due to skew symmetry. The coefficients $c_1$ and $c_3$ are given in appendix \ref{c1c3}.
\begin{figure}
\begin{center}
\includegraphics[scale=0.35]{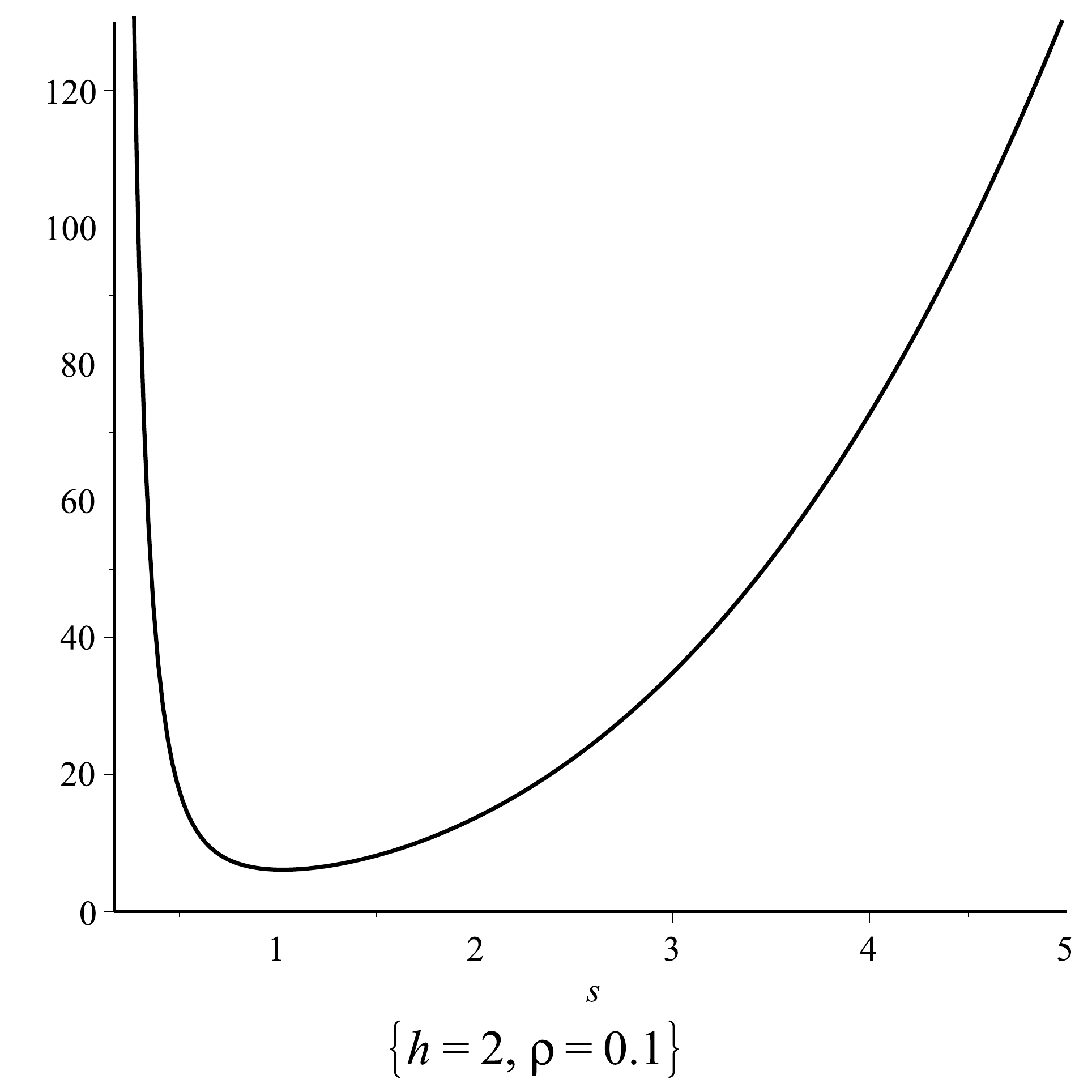}
\includegraphics[scale=0.35]{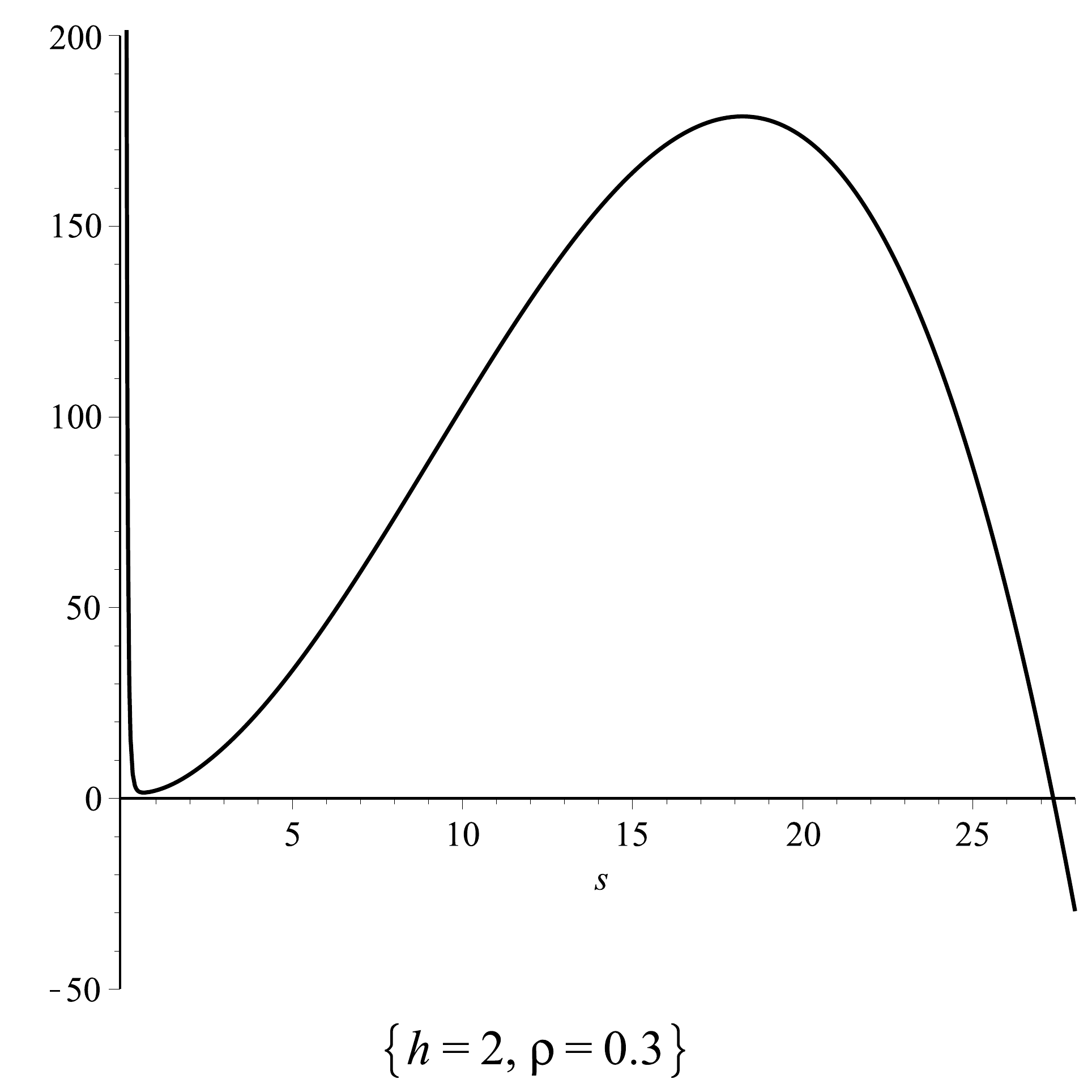}
\includegraphics[scale=0.35]{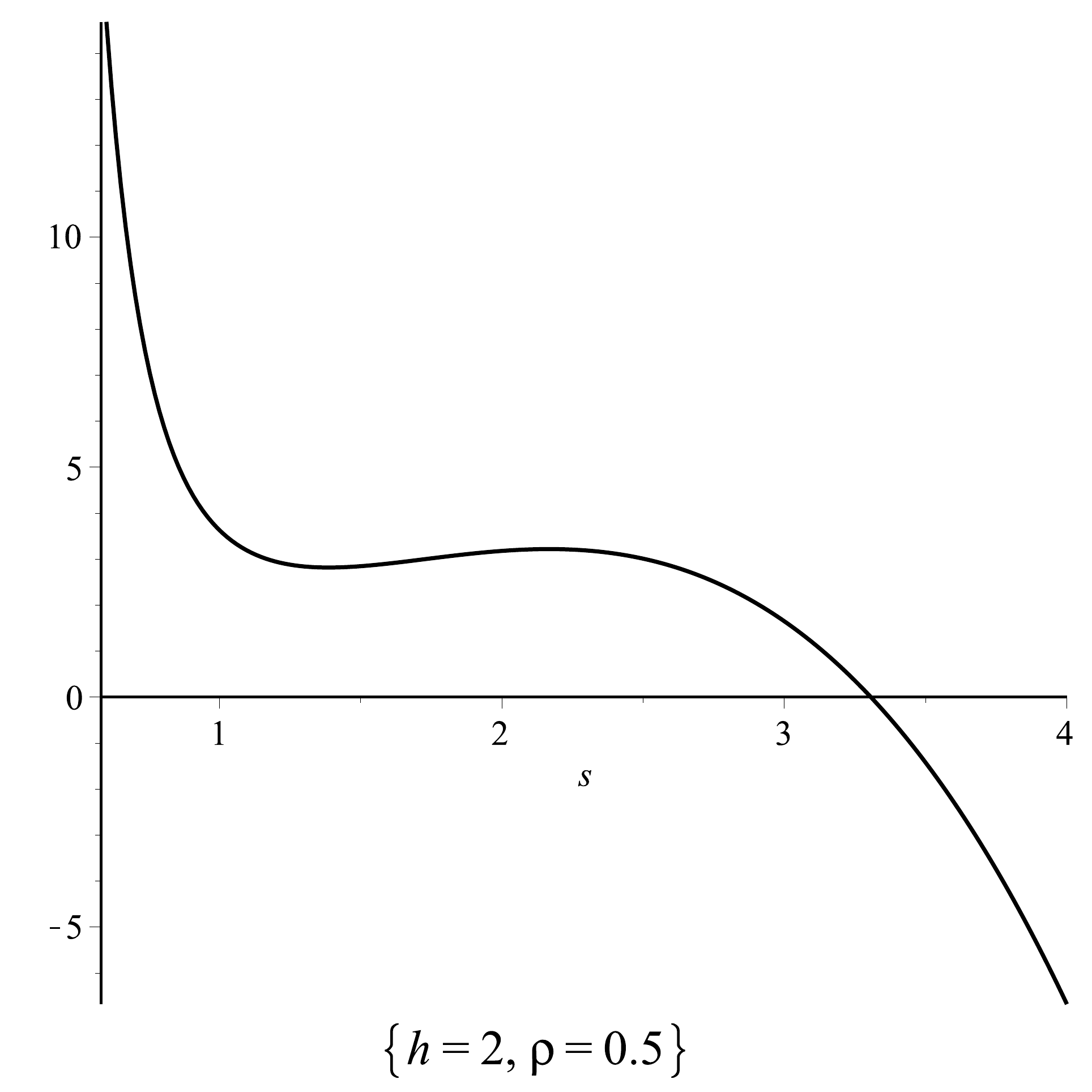}
\includegraphics[scale=0.35]{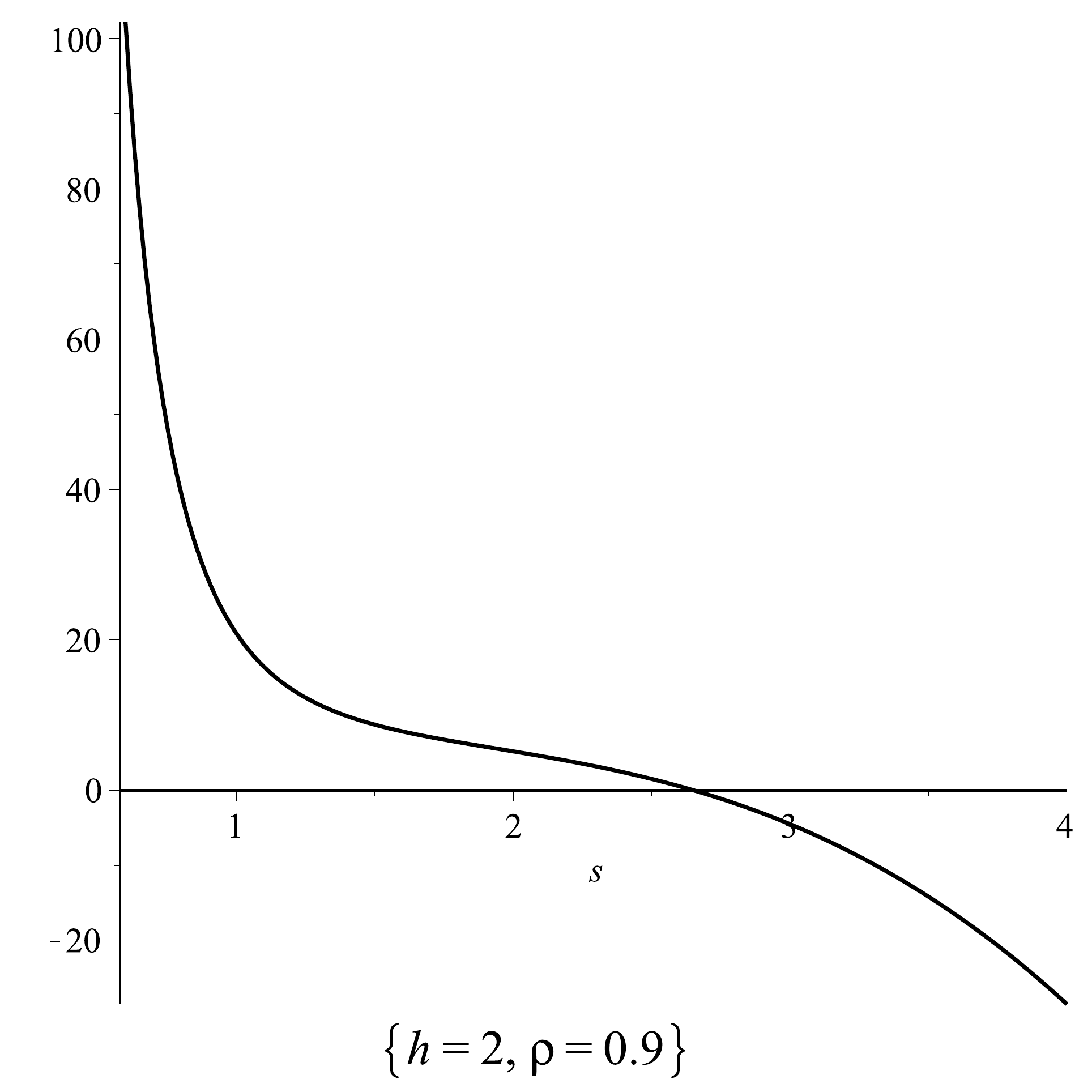}
\caption{Plots of $2\gamma_1c_3(k)$, as functions of $k$ for increasing values of $\rho$ and where $h$ is kept fixed. It seems to be the case that for small $\rho$, $c_3(k)>0$ for all $k>0$. This agrees with the results found when Laurent expanding $c_3(k)$ (see appendix \ref{c1c3}). However, when $\rho$ is increased it seems that $c(k)=0$ exactly once.}\label{rhoplot}
\end{center}
\end{figure}
\begin{figure}
\begin{center}
\includegraphics[scale=0.35]{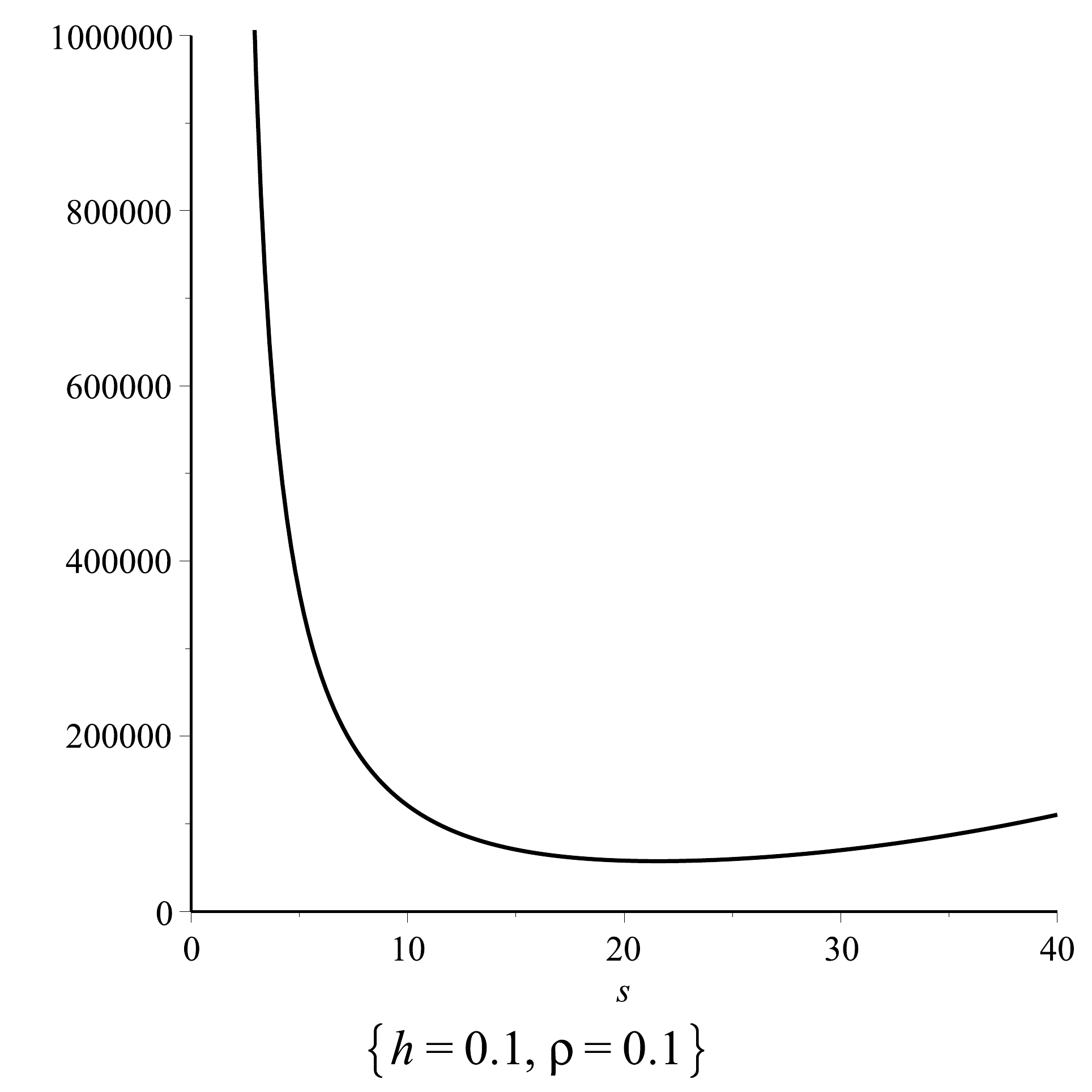}
\includegraphics[scale=0.35]{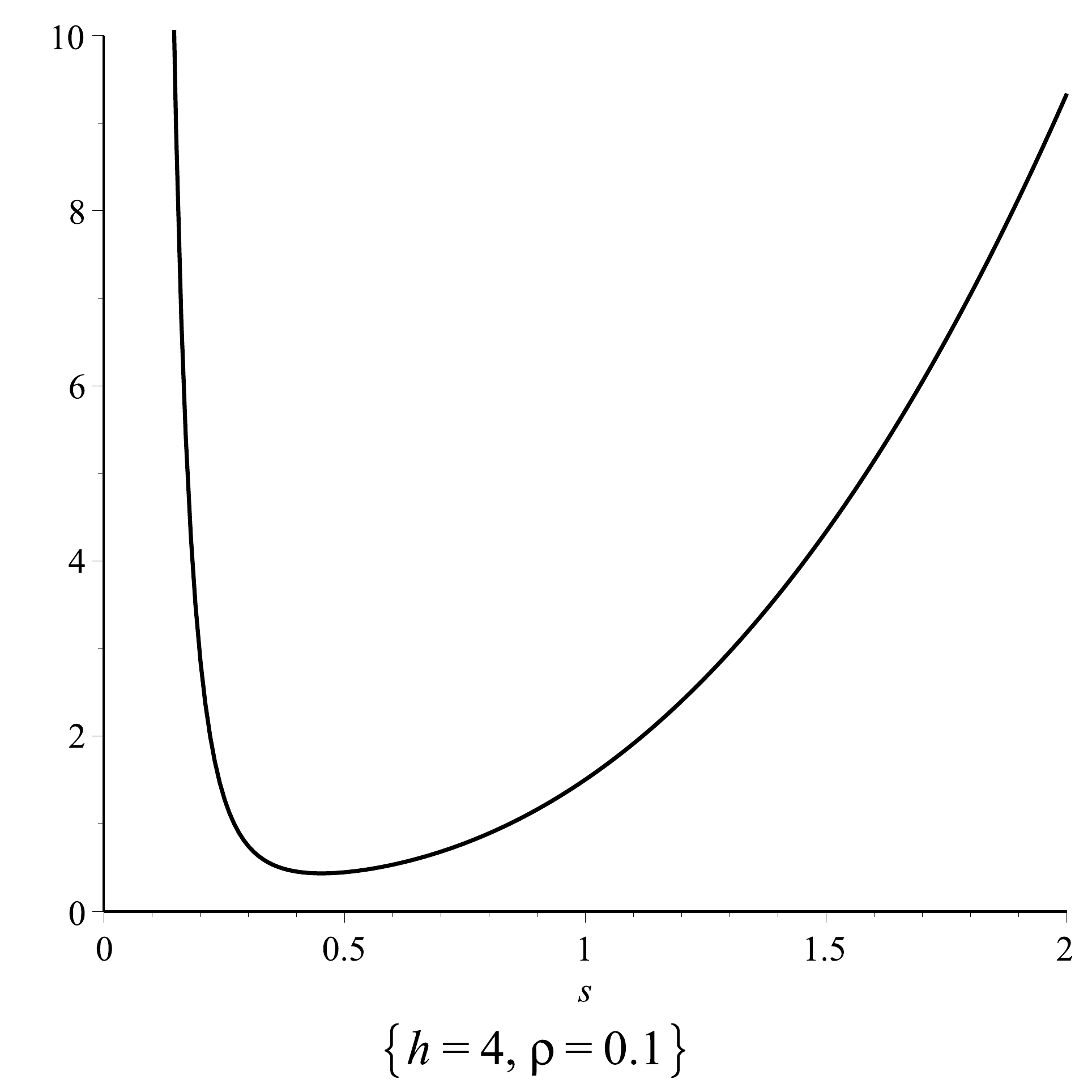}
\includegraphics[scale=0.35]{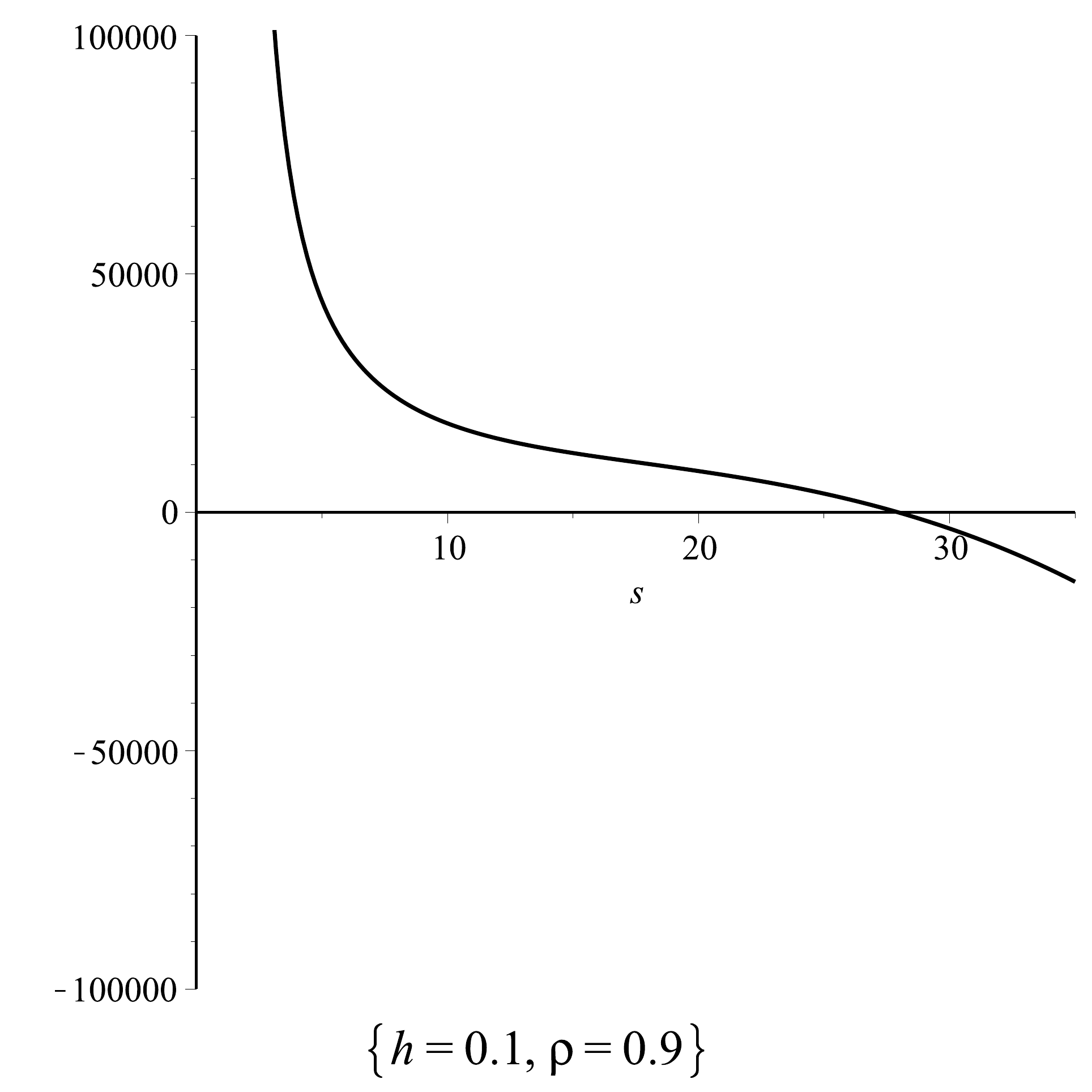}
\includegraphics[scale=0.35]{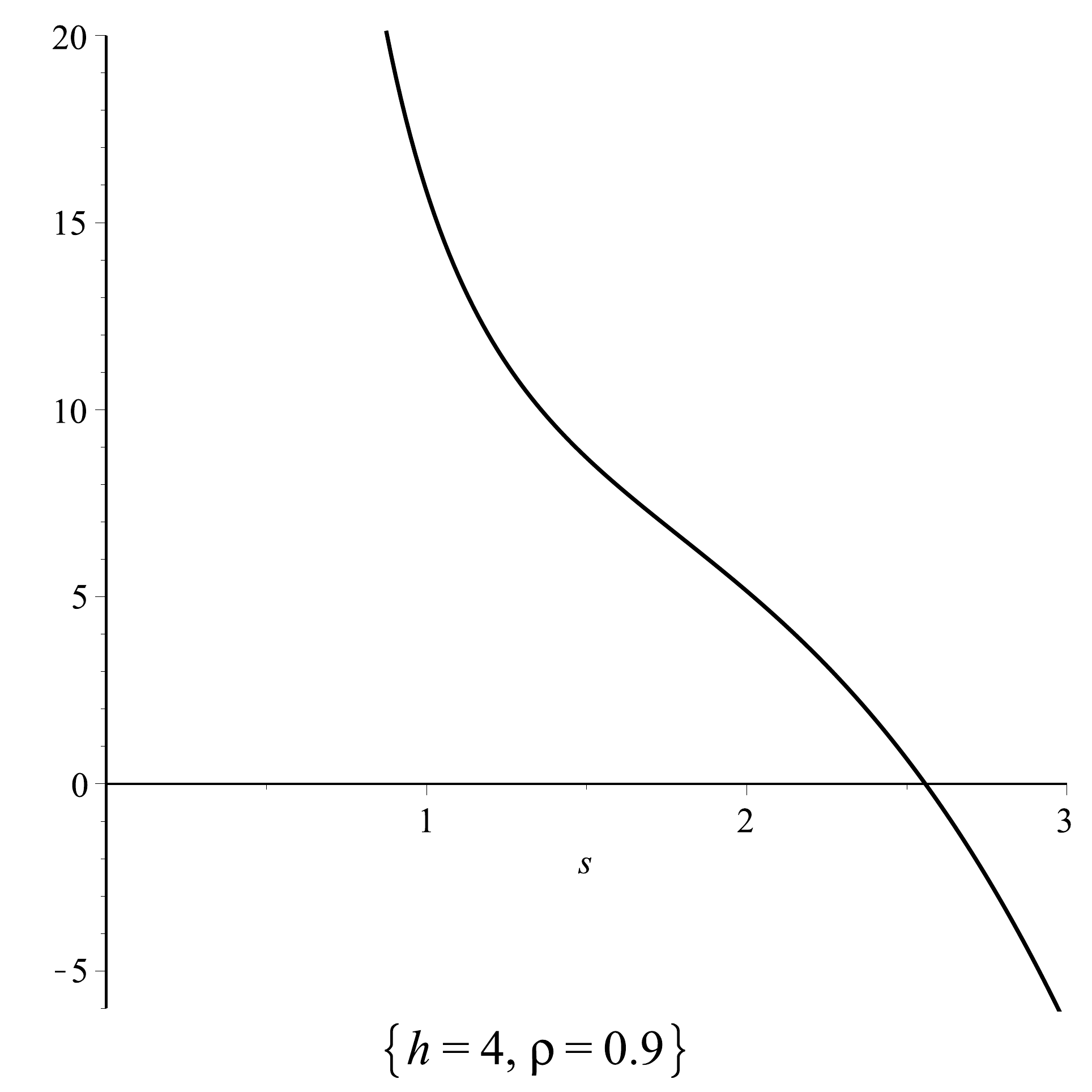}
\caption{Plots of $2\gamma_1c_3(k)$ as functions of $k$ for $h=0.1$, $h=4$ and for $\rho=0.1$, $\rho=0.9$. It seems that $h$ does not affect the overall profile of the graph.}\label{hplot}%kanske lägg till någon mer kommentar här
\end{center}
\end{figure}
\begin{figure}
\centering
\includegraphics{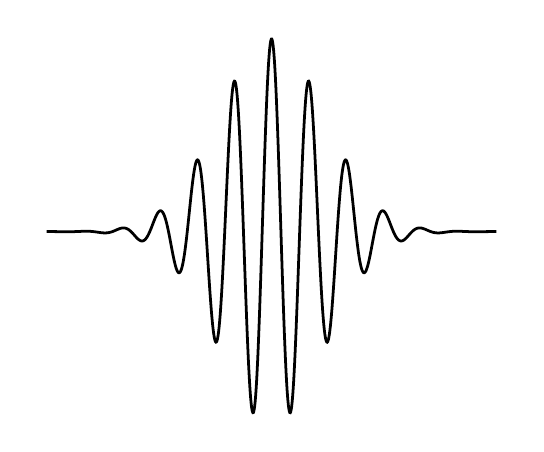}
\includegraphics{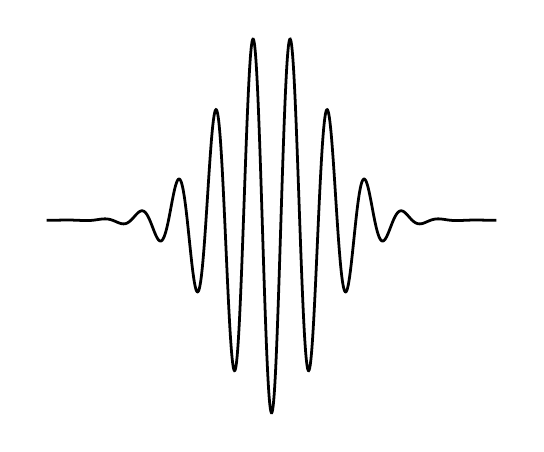}
\includegraphics{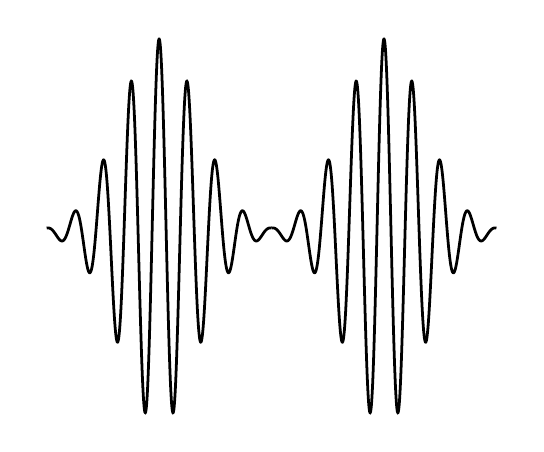}
\includegraphics{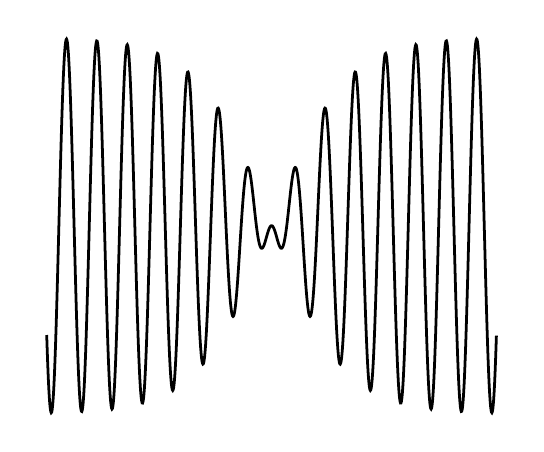}
\caption{Typical wave profiles of the different solutions found from theorem \ref{caseshamhopf}. From left to right: Bright solitary wave of elevation, bright solitary wave of depression, multipulse bright solitary wave of elevation, dark solitary wave.}
\label{darkbright}
\end{figure}
In our case we have $c_1<0$, and from figures \ref{rhoplot} and \ref{hplot} we see that $k,\ \rho$ and $h$ can be chosen so that $c_3<0$.  Hence, all of the above cases can occur if we allow $\delta$ to be negative. Moreover, it seems to be the case that $c_3$ can change sign at most once as a function of $k$. By doing a Laurent series expansion, we find that $c_3(k)>0$ for all $\rho\in [0,1)$, $h\in(0,\infty)$ and $k$ sufficiently small (see appendix \ref{c1c3}). Hence, when $k$ is small we do not find any dark solitary wave solutions. This is illustrated in figure \ref{fig2} where region $II$ have been divided into two parts. The region with the lighter shade is where there exist bright solitary waves and the region with the darker shade is where there exist dark solitary waves.

The plots in figure \ref{rhoplot} indicate that $c_3$ is positive when $\rho$ is small. This agrees with the result found in \cite{DI2} and \cite{BG}, where the authors calculate the corresponding coefficient for the surface wave case ($\rho=0$) and find that this is positive for all $k>0$. 

The Hamiltonian-Hopf bifurcation has been considered in infinite depth as well by Dias and Iooss in \cite{DI}. The authors found the solutions given in theorem \ref{caseshamhopf} and they also found a critical value of $\rho$ where $c_3$ change sign. When using the same non-dimensionalization as in \cite{DI} and letting $h_1$ and $h_2$ tend to infinity we find that 
\begin{equation*}
c_3\rightarrow Dc_3^*
\end{equation*}
where $D$ is some constant and $c_3^*$ is the coefficient for the term $A\abs{A}^2$ found in \cite{DI}. However, this is only a formal argument and does not imply that the solutions from theorem \ref{caseshamhopf} should persist when $h_1$ and $h_2$ tend to infinity.
Finally, we see from figures \ref{rhoplot} and \ref{hplot} that there exists values of $\rho$, $h$ and $k$ for which $c_3=0$. For such values one has to compute higher order terms of the normal form.
%\newpage

\subsection{Hamiltonian real 1:1 resonance}\label{11}
In this section we study the real $1:1$ resonance which occurs when crossing the curve $C_1^{\rho,h}$ in a small neighborhood around the point $(\beta_0,\alpha_0)$ (see region $I$ in figure \ref{fig2}). The Taylor expansions of $\hat{\alpha}(k)$ and $\hat{\beta}(k)$ are given by
\begin{align*}
\hat{\beta}(k)&=\beta_0+2\gamma_3k^2+\mathcal{O}(k^4),\\
\hat{\alpha}(k)&=\alpha_0+\gamma_3k^4+\mathcal{O}(k^6).
\end{align*}
where $\gamma_3=(\rho+h^3)/45$.
In order to study this case using the center manifold theorem, we let $(\beta,\alpha)=(\beta_0,\alpha_0)+\mu$, where $\mu=(\mu_1,\mu_2)=(2(1+\delta)\gamma_3\epsilon^2,\gamma_3\epsilon^4)$ for some $\delta<0$.
From section \ref{parameterregimes} we get that for the case when $\alpha=\alpha_0$, $\beta>\beta_0$, $0$ is an eigenvalue of algebraic multiplicity 2. Moreover, when $\alpha=\alpha_0$ and $\beta= \beta_0$, then $0$ is an eigenvalue of algebraic multiplicity 4, with generalized eigenvectors
\begin{equation}
e_1=\left( \begin{array}{c}
1\\
0\\
0\\
0\\
0\\
0
\end{array}\right)
,\quad e_2=\left( \begin{array}{c}
0\\
0\\
\frac{z^2-\frac{1}{3}}{2}\\
0\\
-\frac{h(z^2-\frac{1}{3})}{2}\\
0
\end{array}\right)
,\quad e_3=\left( \begin{array}{c}
0\\
0\\
0\\
\frac{\rho z(z^2-1)}{6}\\
0\\
-\frac{h^2z(z^2-1)}{6}
\end{array}\right)
,\quad e_4= \left( \begin{array}{c}
0\\
-\frac{\rho+h^3}{45}\\
-\frac{z^2}{12}\left(\frac{z^2}{2}-1\right)-\frac{7}{360}\\
0\\
\frac{h^3z^2}{12}\left(\frac{z^2}{2}-1\right)+\frac{7h^3}{360}\\
0
\end{array}\right),\label{geneigen04}
\end{equation}
such that $Le_1=0,\ Le_2=e_1,\ Le_3=e_2 $ and $Le_4=e_3$.
These vectors satisfy
\begin{align*}
&\Omega(e_1,e_2)=0,\ \Omega(e_1,e_3)=0,\ \Omega(e_1,e_4)=-\gamma_3,\\
&\Omega(e_2,e_3)=\gamma_3,\ \Omega(e_2,e_4)=0,\ \Omega(e_3,e_4)=-\frac{2(\rho+h^5)}{945}:=-\gamma_4.
\end{align*}
The vectors
\begin{equation}
W_1=\frac{e_4}{\sqrt{\gamma_3}},\ W_2=\frac{e_1}{\sqrt{\gamma_3}},\ W_3=\frac{e_2}{\sqrt{\gamma_3}},\ W_4=\frac{e_3-\frac{\gamma_4e_1}{\gamma_3}}{\sqrt{\gamma_3}},\label{symplectic04}
\end{equation}
satisfy $\Omega(W_1,W_2)=\Omega(W_3,W_4)=1$ and all other combinations are equal to 0. Hence, the vectors $W_i$, $i=1,\ldots ,4$ form a symplectic basis for the vector space spanned by $e_i$, $i=1,\ldots,4$.
As in the previous sections we find that $f_i=dG[0](W_i)$ is a symplectic basis of $E_1=P(\tilde{M})$, where the vectors $W_i$, $i=1,\ldots, 4$ are given in \eqref{symplectic04}. 
We apply the center manifold theorem together with Darboux's theorem and obtain a Hamiltonian system $(X_C,\Psi,\tilde{H}^\mu)$, where $X_C=\{u_1+r(u_1,\mu)\ :\ u_1\in \tilde{U}_1\}$ and $\tilde{U}_1$ is a neighborhood of $0$. Every element $u_1$ in $E_1$ can be written as 
\begin{equation*} 
u_1=(q_1,p_1,q_2,p_2)=q_1f_1+p_1f_2+q_2f_3+p_2f_4,
\end{equation*}
The vectors $f_i$, $i=1,\ldots,4$ are given by
\begin{alignat*}{2}
&f_1=\frac{1}{\sqrt{\gamma_3}}\left(\begin{array}{c}
0\\
\gamma_3\\
\rho\left(-\frac{z^2}{12}(\frac{z^2}{2}-1)-\frac{7}{360}\right)\\
0\\
\frac{h^3z^2}{12}(\frac{z^2}{2}-1)+\frac{7h^3}{360}\\
0
\end{array}\right)
, &&\quad f_2=\frac{1}{\sqrt{\gamma_3}}\left(\begin{array}{c}
1\\
0\\
0\\
0\\
0\\
0
\end{array}\right)
,\\
&f_3=\frac{1}{\sqrt{\gamma_3}}\left(\begin{array}{c}
0\\
\beta_0\\
0\\
0\\
0\\
0
\end{array}\right),
&&\quad f_4=\frac{1}{\sqrt{\gamma_3}}\left(\begin{array}{c}
-\frac{\gamma_4}{\gamma_3}\\
0\\
0\\
\frac{\rho z(z^2-1)}{6}\\
0\\
-\frac{h^2z(z^2-1)}{6}
\end{array}\right),
\end{alignat*}
with $\Psi(f_1,f_2)=\Psi(f_3,f_4)=1$ and every other combination is equal to $0$.
We begin by calculating the terms of order $\abs{u_1}^2$.
The transformation $S$ is given in coordinates by
\begin{equation*}
S:(q_1,p_1,q_2,p_2)\rightarrow (-q_1,p_1,-q_2,p_2),
\end{equation*}
which means that $\tilde{H}^\mu(-q_1,p_1,-q_2,p_2)=\tilde{H}^\mu(q_1,p_1,q_2,p_2)$.
From this we conclude that the coefficients of $q_1p_1,\ q_1p_2,\ q_2p_1$ and $q_2p_2$ are all equal to zero. The remaining coefficients are found by calculating
\begin{alignat*}{2}
H_2^{00}(f_1,f_1)&=-\frac{\gamma_4}{2\gamma_3},&&\quad H_2^{00}(f_2,f_2)=0,\\
H_2^{00}(f_3,f_3)&=0, &&\quad H_2^{00}(f_4,f_4)=\frac{1}{2},\\
H_2^{00}(f_1,f_3)&=-1,&&\quad H_2^{00}(f_2,f_4)=0,
\end{alignat*}   
which means that
\begin{equation*}
\tilde{H}_2^{00}(u_1)=-\frac{\gamma_4q_1^2}{2\gamma_3}-q_1q_2+\frac{p_2^2}{2}.
\end{equation*}
As in section \ref{hamhopf} we use normal form theory in order to compute higher order terms of the Hamiltonian. Using this we obtain (see \cite{EL}, \cite{MH})
\begin{equation*}
\tilde{H}^\mu(q_1,p_1,q_2,p_2)=\tilde{H}_2^{00}(q_1,p_1,q_2,p_2)+\hat{H}^\mu(I_1,I_2,I_3,I_4)+\mathcal{O}(\abs{u_1}^3\abs{\mu})+\mathcal{O}(\abs{u_1}^5),
\end{equation*}
where 
\begin{align*}
I_1&=p_1,\\
I_2&=q_2^2-2p_1p_2,\\
I_3&=q_2^3+3p_1^2q_1-3p_1p_2q_2-\frac{3\gamma_4}{\gamma_3}p_1^2q_2,\\
I_4&=-\frac{4\gamma_4}{\gamma_3}p_1^2p_2^2+\frac{\gamma_4}{\gamma_3}p_1p_2q_2^2-\frac{4}{3}p_1p_2^3+\frac{1}{2}p_2^2q_2^2-\frac{3}{2}p_1^2q_1^2\\
&\quad -\frac{3\gamma_4^2}{\gamma_3^2}p_1^3p_2-q_1q_2^3+3p_1p_2q_1q_2+\frac{3\gamma_4}{\gamma_3}p_1^2q_1q_2,
\end{align*}
and $\hat{H}$ is a polynomial function of the variables $I_1,I_2,I_3$,$I_4$ and of degree $4$ in $q_1,p_1,q_2,p_2$, with no constant or first order terms. In particular we find that
\begin{align}
\tilde{H}_3^{00}(u_1)&=c_1^{00}p_1^3+c_2^{00}p_1(q_2^2-2p_1p_2)\label{o3}\\
\tilde{H}_2^{10}(u_1)&=c_1^{10}\mu_1p_1^2+c_2^{10}\mu_1(q_2^2-2p_1p_2)\label{omu12}\\
\tilde{H}_2^{01}(u_1)&=c_1^{01}\mu_2p_1^2+c_2^{01}\mu_2(q_2^2-2p_1p_2)\label{omu22}\\
\tilde{H}_2^{20}(u_1)&=c_1^{20}\mu_1^2p_1^2+c_2^{20}\mu_1^2(q_2^2-2p_1p_2).\label{omu122}
\end{align}
We introduce the scaling $q_1(x)=\epsilon^7Q_1(X),\  p_1(x)=\epsilon^4P_1(X),\ q_2(x)=\epsilon^5Q_2(X),\ p_2(x)=\epsilon^6P_2(X)$. If we write
\begin{align}
\tilde{H}^\mu(q_1,p_1,q_2,p_2)&=-q_1q_2+\frac{p_2^2}{2}+c_1^{10}\mu_1p_1^2+c_2^{10}\mu_1(q_2^2-2p_1p_2)\nonumber \\
&\quad +c_1^{01}\mu_2p_1^2+c_1^{20}\mu_1^2p_1^2+c_1^{00}p_1^3+\mathcal{R}(q_1,p_1,q_2,p_2),\label{scaledham}
\end{align}
we see that every term is of order $\epsilon^{12}$, except the $\mu_1p_1^2$ term which is of order $\epsilon^{10}$, and $\mathcal{R}(q_1,p_1,q_2,p_2)$ is of order $\epsilon^{14}$. The idea is therefore to find the coefficients in \eqref{scaledham}.

Consider now terms of order $\abs{u_1}\mu_1$ in \eqref{hamfullmod}. We then get the following equations:
\begin{align*}
p_1\mu_1&:Kr_{0100}^{10}=-F_1^{10}(f_2)+2c_1^{10}f_1-2c_2^{10}f_3,\\
q_2\mu_1&:Kr_{0010}^{10}-r_{0100}^{10}=-F_1^{10}(f_3)-2c_2^{10}f_4,\\
p_2\mu_1&:Kr_{0001}^{10}-r_{0010}^{10}=-F_1^{10}(f_4)-2c_2^{10}f_1,
\end{align*} 
and calculations show that 
\begin{equation*}
F_1^{10}(f_i)=0, \ i=1,2,3,\quad F_1^{10}(f_4)=\frac{1}{2\sqrt{\gamma_3}\beta_0}\left(\begin{array}{c}
0\\
0\\
\rho(z^2-\frac{1}{3})\\
0\\
-h(z^2-\frac{1}{3})\\
0
\end{array}\right),
\end{equation*}
which implies that
\begin{align}
Kr_{0100}^{10}&=2c_1^{10}f_1-2c_2^{10}f_3\label{eq1102},\\
Kr_{0010}^{10}-r_{0100}^{10}&=-2c_3^{10}f_4\label{eq1103},\\
Kr_{0001}^{10}-r_{0010}^{10}&=-F_1^{10}(f_4)-2c_2^{10}f_1\label{eq1104}.
\end{align} 
Recall that $f_1=dG[0](W_1)$ and $W_1=\frac{e_4}{\sqrt{\gamma_3}}$. This implies that \eqref{eq1102} has no solutions unless 
\begin{equation*}
c_1^{10}=0.
\end{equation*}
It follows directly from this choice of $c_1^{10}$ that $r_{0100}^{10}=-2c_2^{10}f_4+Af_2$, where $A$ is a constant. Consider next \eqref{eq1103}:
\begin{equation*}
Kr_{0010}^{10}=Af_2-4c_2^{10}f_4.
\end{equation*}
The solution of this equation is given by
\begin{equation}
r_{0010}^{10}=-4c_2^{10}f_1+Bf_2+Cf_3,\label{r001010}
\end{equation}
for some constants $B$ and $C$.
This is used in \eqref{eq1104} to obtain
\begin{equation*}
Kr_{0001}^{10}=-6c_2^{10}f_1-F_1^{10}(f_4)+Bf_2+Cf_3,
\end{equation*}
and this equation is solvable if and only if 
\begin{equation*}
c_2^{10}=-\frac{1}{6\gamma_3}.
\end{equation*}
The coefficient $c_1^{01}$ is found from the calculation
\begin{align*}
c_1^{01}&=H_2^{01}[f_2,f_2]+2H_2^{00}[f_2,r_{0100}^{01}]\\
&=H_2^{01}[f_2,f_2]=-\frac{1}{2\gamma_3},
\end{align*}
where we used \eqref{skewsymmetry} to conclude that $2H_2^{00}[f_2,r_{0100}^{01}]=\Psi(Kf_2,r_{0100}^{01})=0$.

Consider next terms of order $p_1\mu_1^2$ in \eqref{hamfullmod}.% In order to do this we use that
%\begin{equation*}
%\tilde{H}_2^{20}(u_1)=\sum_{k+l+m+n=2}c_{klmn}^{20}\mu_1^2q_1^kp_1^lq_2^mp_2^n,
%\end{equation*}
%which implies that
%\begin{align*}
%N_1^{20}&=(2c_{0200}^{20}p_1\mu_1^2+c_{0101}^{20}p_2\mu_1^2)f_1-(2c_{2000}^{20}q_1\mu_1^2+c_{1010}^{20}q_2\mu_1^2)f_2\\
%&\quad +(2c_{0002}^{20}p_2\mu_1^2+c_{0101}^{20}p_1\mu_1^2)f_3-(2c_{0020}^{20}q_2\mu_1^2+c_{1010}^{20}q_1\mu_1^2)f_4.
%\end{align*}
%When considering terms of order $p_1\mu_1^2$ \eqref{hamfullmod}, we obtain
\begin{equation}
Kr_{0100}^{20}=-F_1^{20}(f_2)-F_1^{10}(r_{0100}^{10})+2c_1^{20}f_1-2c_2^{20}f_3-2c_2^{10}r_{0010}^{10}\label{c020020}.
\end{equation}
We find that $F_1^{20}(f_2)=0$ and since $r_{0100}^{10}=-2c_2^{10}f_1+Af_2$, we have that $F_1^{10}(r_{0100}^{10})=-2c_2^{10}F_1^{10}(f_4)$. 
%When taking the symplectic product with $f_2$ on both sides in \eqref{c020020} we obtain
%\begin{equation*}
%c_1^{20}=\frac{1}{2}\Psi(Kr_{0100}^{20},f_2)=-\frac{1}{2}\Psi(r_{0100}^{20},Kf_2)=0.
%\end{equation*}
%förklara skew-symmetry%
Using this and the expression for $r_{0010}^{10}$ given in \eqref{r001010}, we find that \eqref{c020020} is solvable if and only if 
\begin{equation*}
c_1^{20}=\frac{1}{18\gamma_3^2}.
\end{equation*}
Finally we calculate $c_1^{00}$. This is done by considering the $p_1^2$ component of \eqref{hamfullmod}:
\begin{align}
Kr_{0200}^{00}&=-F_2^{00}(f_2,f_2)+3c_1^{00}f_1-2c_2^{00}f_3\label{c0300}\Rightarrow\\
c_1^{00}&=\frac{\Psi(F_2^{00}(f_2,f_2),f_2)}{3}=\frac{(\rho-\frac{1}{h^2})}{2\gamma_3^{\frac{3}{2}}}.\nonumber
\end{align}
The variables are scaled in the following way:
\begin{equation*}
X=\epsilon x,\quad q_1(x)=\epsilon^7Q_1(X),\quad p_1(x)=\epsilon^4P_1(X),\quad q_2(x)=\epsilon^5Q_2(X),\quad p_2(x)=\epsilon^6P_2(X).
\end{equation*}
Then, from \eqref{scaledham}
\begin{align*}\tilde{H}^\mu(Q_1,P_1,Q_2,P_2)&=\frac{1}{2}\epsilon^{12}P_2^2-\epsilon^{12}Q_1Q_2-\frac{\mu_1\epsilon^{10}}{6\gamma_3}\left(Q_2^2-2P_1P_2\right)\\
&\quad -\frac{\mu_2}{2\gamma_3}\epsilon^8P_1^2+\frac{1}{18\gamma_3^2}\mu_1^2\epsilon^8P_1^2+\frac{\rho -\frac{1}{h^2}}{2\gamma_3^\frac{3}{2}}P_1^3+\mathcal{O}(\epsilon^{14})\\
&=\epsilon^{12}\left(\frac{P_2^2}{2}-Q_1Q_2-\frac{1+\delta}{3}(Q_2^2-2P_1P_2)-\frac{P_1^2}{2}+\frac{2(1+\delta)^2P_1^2}{9}+\frac{(\rho -\frac{1}{h^2})P_1^3}{2\gamma_3^\frac{3}{2}}\right)\\
&\quad+\mathcal{O}(\epsilon^{14}).
\end{align*}
It follows that Hamilton's equations are given by
\begin{align}
Q_{1X}&=-P_1+\frac{3P_1^2(\rho -\frac{1}{h^2})}{2\gamma_3^\frac{3}{2}}+\frac{2(1+\delta)P_2}{3}+\frac{4(1+\delta)^2P_1}{9}+\mathcal{O}(\epsilon)\label{fsys1},\\
P_{1X}&=Q_2+\mathcal{O}(\epsilon)\label{fsys2},\\
Q_{2X}&=P_2+\frac{2(1+\delta)P_1}{3}+\mathcal{O}(\epsilon)\label{fsys3},\\
P_{2X}&=Q_1+\frac{2(1+\delta)Q_2}{3}+\mathcal{O}(\epsilon)\label{fsys4}.
\end{align}
Hence, when $\epsilon\rightarrow 0$, the above system is equivalent to the fourth order ODE
\begin{equation}
\partial_X^4P_1-2(1+\delta)\partial_X^2P_1+P_1-\frac{3(\rho-\frac{1}{h^2})P_1^2}{2\gamma_3^\frac{3}{2}}=0\label{ode1}.
\end{equation}
Let $w=3(\rho -1/h^2)/(2\gamma_3^\frac{3}{2})P_1$. Then equation \eqref{ode1} becomes 
\begin{equation}
\partial_X^4w-2(1+\delta)\partial_X^2w+w-w^2=0 \label{ode2}.
\end{equation}
This equation has been studied in \cite{BCT}. The authors showed that for $\delta= 0$, equation \eqref{ode2} has a non-trivial solution which is unique up to translation. The orbit of this solution lies in the transverse intersection, with respect to the zero energy manifold, of the stable and unstable manifolds of the origin. By the stable manifold theorem this also holds for \eqref{fsys1}--\eqref{fsys4}, for sufficiently small $\epsilon$ and $\delta<0$. We call this the primary homoclinic solution.  For $\delta<0$ the equilibrium solution $0$ is a saddle point of \eqref{fsys1}--\eqref{fsys4}. It is shown in \cite{DV} that this implies the existence of a countable family of solutions which resemble multiple copies of the primary homoclinic solution. The existence of such solutions is also discussed in \cite{BCT} for equation \eqref{ode2}. If $P_1$ is a solution of \eqref{ode2} then, in original variables, the wave profile is given by
\begin{equation*}
\eta(x)=h_1\epsilon^4\sqrt{\gamma_3}P_1\left(\frac{\epsilon x}{h_1}\right)+ \mathcal{O}(\epsilon^5).
\end{equation*}
These solutions correspond to multi-troughed solitary waves of depression when $\rho-1/h^2<0$ and of elevation when $\rho-1/h^2>0$.

Consider next the case when $\rho-1/h^2$ is small. For this case we include the term of order $p_1^4$ in the Taylor expansion of the reduced Hamiltonian. We have that
\begin{equation*}
\tilde{H}_4^{00}(u_1)=e_1^{00}p_1^4+e_2^{00}(q_2-2p_1p_2)^2+e_3^{00}p_1^2(q_2^2-2p_1p_2)+e_4^{00}I_4
\end{equation*}
In order to find the coefficient $e_1^{00}$  the $p_1^3$ component of \eqref{hamfullmod} is considered.
\begin{align*}
Kr_{0300}^{00}&=-F_3^{00}(f_2,f_2,f_2)-2F_2^{00}(f_2,r_{0200}^{00})+4e_1^{00}f_1-(2e_3^{00}+\frac{3\gamma_4^2e_4^{00}}{\gamma_3^2})f_3\\
&\quad -2c_2^{00}r_{0110}^{00}+3c_1^{00}r_{1100}^{00}.
\end{align*}
When taking the symplectic product with $f_2$ of both sides of the above equation, we obtain
\begin{equation}
4e_1^{00}=\Psi(F_3^{00}(f_2,f_2,f_2),f_2)+2\Psi(F_2^{00}(f_2,r_{0200}^{00}),f_2)+2c_2^{00}\Psi(r_{0110}^{00},f_2)-3c_1^{00}\Psi(r_{1100}^{00},f_2)\label{c0400}.
\end{equation}
Hence, we need to find $r_{0200}^{00}$, $r_{0110}^{00}$ and $c_2^{00}$. It won't be necessary to calculate $c_1^{00}\Psi(r_{1100}^{00},f_2)$, since we assume that $\rho-1/h^2$ is small and $c_1^{00}=\mathcal{O}(\rho-1/h^2)$. Consider \eqref{c0300}
\begin{equation*}
Kr_{0200}^{00}=-F_2^{00}(f_2,f_2)+3c_1^{00}f_1-2c_2^{00}f_3.
\end{equation*}
This equation is solvable with solution
\begin{equation*}
r_{0200}^{00}=Z_1-2c_2^{00}f_4+Af_2,
\end{equation*}
where $A$ is an arbitrary constant and 
\begin{equation*}
Z_1=\frac{3(\rho-\frac{1}{h^2})}{2\gamma_3}\left(\begin{array}{c}
0\\
0\\
0\\
\frac{\rho}{12\gamma_3}\left(-\frac{z^5}{10}+\frac{z^3}{3}-\frac{7z}{30}\right)\\
0\\
\frac{h^4}{12\gamma_3}\left(\frac{z^5}{10}-\frac{z^3}{3}+\frac{7z}{30}\right)
\end{array}\right),
\end{equation*}
is a solution of the equation 
\begin{equation*}
Ku=-F_2^{00}(f_2,f_2)+3c_1^{00}f_1.
\end{equation*}
The $p_1q_2$ component of \eqref{hamfullmod}, is given by
\begin{equation}
Kr_{0110}^{00}-2r_{0200}^{00}=-2F_2^{00}(f_2,f_3)-2c_2^{00}f_4 \label{eqr0110}.
\end{equation}
and $F_2^{00}(f_2,f_3)=0$, so the solution of \eqref{eqr0110} is given by
\begin{equation}
r_{0110}^{00}=Z_2-6c_2^{00}f_1+Bf_2+Cf_3\label{r0110},
\end{equation}
where $B$ and $C$ are arbitrary constants and
\begin{equation*}
Z_2=\frac{\rho-\frac{1}{h^2}}{\gamma_3^2}\left(\begin{array}{c}
0\\
3\gamma_4\\
\frac{\rho}{48}\left[\frac{z^6}{5}-z^4+\frac{7z^2}{5}-\frac{31}{105}\right]\\
0\\
-\frac{h^5}{48}\left[\frac{z^6}{5}-z^4+\frac{7z^2}{5}-\frac{31}{105}\right]\\
0
\end{array}\right),
\end{equation*}
is a solution of the equation
\begin{equation*}
Ku=2Z_1.
\end{equation*}
When considering terms of order $p_1p_2$ in \eqref{hamfullmod} we get
\begin{align*}
Kr_{0101}^{00}&=r_{0110}^{00}-2F_2^{00}(f_2,f_4)-4c_2^{00}f_1\\
&=Z_2-2F_2^{00}(f_2,f_4)-10c_2^{00}f_1+Bf_2+Cf_3
\end{align*}
and $Z_2-2F_2^{00}(f_2,f_4)-10c_2^{00}f_1\in \mathcal{R}(K)$ if and only if
\begin{equation}
c_2^{00}=\frac{1}{5\gamma_3^{\frac{3}{2}}}\left(\frac{3\left(\rho-\frac{1}{h^2}\right)\gamma_4}{\gamma_3}+\frac{\rho-1}{3}\right)\label{c0201},
\end{equation}
where
\begin{equation*}
F_2^{00}(f_2,f_4)=\frac{1}{2\gamma_3}\left(\begin{array}{c}
\frac{2(\rho-1)}{3}\\
\frac{(3z^2-1)\rho)}{3}-\left(\frac{4(\rho-1)}{3}+\frac{3\gamma_4\left(\rho-\frac{1}{h^2}\right)}{\gamma_3}\right)\frac{\left(z^2-\frac{1}{3}\right)\rho}{2\beta_0}\\
0\\
\frac{3z^2-1}{3}+\left(\frac{4(\rho-1)}{3}+\frac{3\gamma_4\left(\rho-\frac{1}{h^2}\right)}{\gamma_3}\right)\frac{\left(z^2-\frac{1}{3}\right)h}{2\beta_0}\\
0
\end{array}\right).
\end{equation*}
%The $q_1p_1$ component of \eqref{hamfullmod} is given by
%\begin{equation*}
%Kr_{1100}^{00}-\frac{2\gamma_4}{\gamma_3}r_{0200}^{00}=-2F_2^{00}(f_1,f_2)-2c_{2100}^{00}f_2,
%\end{equation*}
%with solution
%\begin{equation}
%r_{1100}^{00}=\frac{\gamma_4}{\gamma_3}Z_2-2Z_3-c_{2100}^{00}f_3+c_3f_2\label{r1100},
%\end{equation}
%where $c_3$ is an arbitrary constant and 
%\begin{equation*}
%Z_3=\frac{1}{2\gamma_3}\left(\begin{array}{c}
%0\\
%-\frac{\rho+h^2}{45}\\
%\frac{\rho(\frac{z^4}{2}-z^2)}{12}+\frac{7\rho}{360}\\
%0\\
%-\frac{h^3(\frac{z^4}{2}-z^2)}{12}-\frac{7h^2}{360}\\
%0
%\end{array}\right),
%\end{equation*}
%is a solution of the equation
%\begin{equation*}
%Ku=F_2^{00}(f_1,f_2).
%\end{equation*}
In conclusion, 
\begin{align*}
r_{0200}^{00}&=Z_1-2c_2^{00}f_4+Af_2,\\
r_{0110}^{00}&=Z_2-6c_2^{00}f_1+Bf_2+Cf_3,
\end{align*}
and $c_2^{00}$ is given in \eqref{c0201}.
This is used in \eqref{c0400}, and it is found that
\begin{equation*}
e_1^{00}=\frac{\rho+\frac{1}{h^3}}{\gamma_3^2}+\frac{2(\rho-1)^2}{225\gamma_3^3}+\left(\rho-\frac{1}{h^2}\right)D,
\end{equation*}
where $D$ is a constant depending on $\rho$ and $h$.
%\begin{align*}
%\tilde{H}^\mu(q_1,p_1,q_2,p_2)&=-\frac{\gamma_4}{2\gamma_3}q_1^2-q_1q_2+\frac{p_2^2}{2}-(1+\delta)\epsilon^2q_2^2 -\frac{\epsilon^4p_1^2}{2}+\frac{(\rho-\frac{1}{h^2})p_1^3}{2\gamma_3^{\frac{3}{2}}}\\
%&\quad -\frac{9(\rho-\frac{1}{h^2})}{16\gamma_3^{\frac{3}{2}}}\left(\frac{\gamma_3}{\beta_0}+\frac{4\gamma_4}{\gamma_3}\right)p_1^2p_2+\left(\frac{\rho+\frac{1}{h^3}}{\gamma_3^2}+\frac{(\rho-1)^2}{450\gamma_3^3}+\left(\rho-\frac{1}{h^2}\right)D\right)p_1^4\\
%&\quad+\mathcal{O}(\abs{q_1,q_2,p_2}^3)+\mathcal{O}(\abs{p_1}\abs{q_1,q_2,p_2}^2)+\mathcal{O}(\abs{p_1^5}).
%\end{align*}
The variables are scaled in the following way:
\begin{equation*}
X=\epsilon x,\ q_1=\epsilon^5Q_1(X),\ p_1(x)=\epsilon^2 P_1(X),\ q_2(x)=\epsilon^3Q_2(X), \  p_2(x)=\epsilon^4P_2(X),
\end{equation*}
and we assume that $\rho-1/h^2=\epsilon^2 \kappa$, where $\kappa$ is a constant which can be both positive and negative. The Taylor expansion of the reduced Hamiltonian becomes
\begin{align*}
\tilde{H}^\mu(Q_1,P_1,Q_2,P_2)&=\epsilon^8\left(-Q_1Q_2+\frac{P_2^2}{2}-\frac{(1+\delta)}{3}(Q_2^2-2P_1P_2)-\frac{P_1^2}{2}\right.\\
&\left.\quad+ \frac{2(1+\delta)^2P_1^2}{9}+\frac{\kappa P_1^3}{2\gamma_3^{\frac{3}{2}}}+\frac{\left(\rho+\frac{1}{h^3}+\frac{2(\rho-1)^2}{225\gamma_3}\right)P_1^4}{\gamma_3^2}\right)+\mathcal{O}(\epsilon^{10}),\label{hammutilde}
\end{align*}
and, Hamilton's equations are given by
\begin{align}
Q_{1X}&=-P_1+\frac{3\kappa P_1^2}{2\gamma_3^{\frac{3}{2}}}+\frac{4\left(\rho+\frac{1}{h^3}+\frac{2(\rho-1)^2}{225\gamma_3}\right)P_1^3}{\gamma_3^2} +\frac{2(1+\delta)P_2}{3}\\
&\quad+\frac{4(1+\delta)^2P_1}{9}+\mathcal{O}(\epsilon),\label{11fulleq1}\\
P_{1X}&=Q_2+\mathcal{O}(\epsilon),\label{11fulleq2}\\
Q_{2X}&=P_2+\mathcal{O}(\epsilon),\label{11fulleq3}\\
P_{2X}&=Q_1+2(1+\delta)Q_2+\mathcal{O}(\epsilon).\label{11fulleq4}
\end{align}
As in the previous case, the truncated system is equivalent to a fourth order ODE:
\begin{equation}
\partial_X^4P_1-2(1+\delta)\partial_X^2P_1+P_1-\frac{3\kappa P_1^2}{2\gamma_3^{\frac{3}{2}}}-\frac{4\left(\rho+\frac{1}{h^3}+\frac{2(\rho-1)^2}{225\gamma_3}\right)P_1^3}{\gamma_3^2}=0 \label{ode23}.
\end{equation}
Let 
\begin{equation*}
u=\frac{2}{\gamma_3}\left(\rho+\frac{1}{h^3}+\frac{2(\rho-1)^2}{225\gamma_3}\right)^\frac{1}{2}P_1,
\end{equation*}
then $u$ satisfies
\begin{equation}
\partial_X^4u-2(1+\delta)\partial_X^2u-f(u)=0\label{ode33},
\end{equation}
where
\begin{equation*}
f(u)=-u+c\kappa u^2+u^3,\quad c=\frac{3}{4\gamma_3^\frac{1}{2}\left[\rho+\frac{1}{h^3}+\frac{2(\rho-1)^2}{225\gamma_3}\right]^\frac{1}{2}}
\end{equation*}
From appendix \ref{B} we know that for $\kappa=0$ and $\delta=0$ there exists a homoclinic solution which corresponds to a transversal intersection of the stable and unstable manifolds in the zero energy manifold. Moreover, we know from appendix \ref{B} that such a solution is either positive or negative. It is clear that if $u$ is a positive solution of \eqref{maineq}, then $-u$ is a negative solution of \eqref{maineq}. Hence, there exists both a positive and a negative solution of \eqref{maineq}. By the stable manifold theorem it follows that for sufficiently small $\delta<0$, $\kappa$ and $\epsilon$, these solutions persist for the full system \eqref{11fulleq4}. Like in the previous case, we may conclude using the theory from \cite{DV}, that for $\delta<0$, $\kappa$ and $\epsilon>0$ sufficiently small, there exist two countable families of homoclinic solutions, each family corresponding to one of the above solutions, of \eqref{11fulleq1}--\eqref{11fulleq4}. In original variables, the wave profile is given by
\begin{equation*}
\eta(x)=h_1\epsilon^2\sqrt{\gamma_3}P_1\left(\frac{\epsilon x}{h_1}\right)+ \mathcal{O}(\epsilon^3).
\end{equation*}
\subsection{$0^2$-resonance}\label{02}
In this section we study the $0^2$ resonance which occurs when crossing the curve $C_3^{\rho,h}$ (see region $III$ in figure \ref{fig2}. We therefore let $(\beta,\alpha)=(\beta,\alpha_0)+(0,\delta)$, where $\beta>\beta_0$. We know from section \ref{parameterregimes} that for these values of $\alpha$ and $\beta$ the imaginary part of the spectrum of $L$ consists of $0$, which is an eigenvalue of algebraic multiplicity $2$ with generalized eigenvectors
\begin{equation}
e_1=\left( \begin{array}{c}
1\\
0\\
0\\
0\\
0\\
0
\end{array}\right)
,\quad e_2=\left( \begin{array}{c}
0\\
\beta-\frac{\rho+h}{3}\\
\frac{z^2-\frac{1}{3}}{2}\\
0\\
-\frac{h(z^2-\frac{1}{3})}{2}\\
0
\end{array}\right)\label{geneigen02}
\end{equation}
such that $Le_1=0$, $Le_2=e_1$ with $\Omega(e_1,e_2)=\beta-\beta_0$.
Let $\gamma_5=\beta-\beta_0$ and
\begin{equation*}
W_1=\frac{e_1}{\sqrt{\gamma_5}},\quad W_2=\frac{e_2}{\sqrt{\gamma_5}}
\end{equation*}
Then $W_1,W_2$ is a symplectic basis for the vector space spanned by $e_1,e_2$. Letting $f_i=dG(0)(W_i)$ we find as in the previous section that $f_i$ is a symplectic basis of $E_1=P(\tilde{M})$. 
We apply the center manifold theorem together with Darboux's theorem and obtain a Hamiltonian system $(X_C^\mu,\Psi,\tilde{H}^\mu)$, where $X_C^\mu=\{u_1+r(u_1,\mu)\ :\ u_1\in \tilde{U}_1\}$ and $\tilde{U}_1$ is a neighborhood of $0$. Every element $u_1$ in $E_1$ can be written as
\begin{equation*}
u_1=(q,p)=qf_1+pf_2,
\end{equation*}
where the vectors $f_1$, $f_2$ are given by 
\begin{equation*}
f_1= \frac{1}{\sqrt{\gamma_5}}\left( \begin{array}{c}
1\\
0 \\
0\\
0\\
0\\
0
\end{array} \right),
\ f_2=\frac{1}{\sqrt{\gamma_5}}\left( \begin{array}{c}
0\\
\frac{\rho+h}{3}\\
0\\
0\\
0\\
0
\end{array}\right),
\end{equation*}
and satisfy $\Psi(f_1,f_2)=1$.  In coordinates $(q,p)$, the transformation $S$ is given by
\begin{equation*}
S(q,p)\rightarrow (q,-p)
\end{equation*}
Hence, we have that $\tilde{H}^\mu(q,-p)=\tilde{H}^\mu(q,p)$.
The Taylor expansion of the reduced Hamiltonian is given by 
\begin{equation*}
\tilde{H}^\mu(q,p)=\sum_{i+j+k=2}^3c_{jk}^i\delta^iq^jp^k+\mathcal{O}(\abs{(\delta,p,q)}^4)
\end{equation*}
and the Taylor expansion of $r$ by
\begin{equation*}
r(q,p,\delta)=\sum_{i+j+k=2}^3r_{jk}^i\delta^iq^jp^k+\mathcal{O}(\abs{(\delta,p,q)}^4)
\end{equation*}
Since $\tilde{H}^\mu$ is invariant under $S$, we can directly conclude that $c_{j1}^i=0$ for all $i,j$. 
We have that
\begin{equation*}
c_{02}^0=H_2^0[f_2,f_2]=\frac{1}{2},\ c_{20}^0=H_2^0[f_1,f_1]=0,
\end{equation*}
and
\begin{align*}
c_{20}^1=H_2^1[f_1,f_1]+2M_2^0[f_1,r_{10}^1]
\end{align*}
where
\begin{equation*}
H_2^1[f_1,f_1]=-\frac{1}{2\gamma_5}
\end{equation*}
and 
\begin{equation*}
2H_2^0[f_1,r_{1,0}^1]=\Psi(Kf_1,r_{10}^1)=0,
\end{equation*}
by \eqref{skewsymmetry} and the identity $Kf_1=0$.
Finally,
\begin{equation*}
c_{30}^0=H_3^0[f_1,f_1,f_1]+2H_2^0[f_1,r_{2,0}^0]=\frac{\rho-\frac{1}{h^2}}{2\gamma_5^{\frac{3}{2}}}.
\end{equation*}
The Taylor expansion of the reduced Hamiltonian is therefore given by
\begin{equation*}
\tilde{H}^\mu(q,p)=\frac{1}{2}p^2-\frac{1}{2\gamma_5}\delta q^2+\frac{\rho -\frac{1}{h^2}}{2\gamma_5^{\frac{3}{2}}}q^3+\mathcal{O}(\abs{p}\abs{(p,q)}\abs{(\delta,p,q)})+\mathcal{O}(\abs{(p,q)}^2\abs{(\delta,p,q)}^2).
\end{equation*}
From this we get Hamilton's equations 
\begin{align}
q_x&=p+\mathcal{O}(\abs{(p,q)}\abs{(\delta,p,q)})\label{hamiltonian_1},\\
p_x&=\frac{\delta q}{\gamma_5}+\frac{3(\frac{1}{h^2}-\rho)q^2}{2\gamma_5^{\frac{3}{2}}}+\mathcal{O}(\abs{p}\abs{(\delta,p,q)})+\mathcal{O}(\abs{(p,q)}\abs{(\delta,p,q)}^2).\label{hamiltonian_2}
\end{align}
Assume first that $1-\rho h^2\neq 0$. 
We make the following change of variables:
\begin{equation*}
X=\frac{\delta^\frac{1}{2}x}{\sqrt{\gamma_5}},\ q(x)=\gamma_5^2\delta Q(X),\ p(x)=\delta^{\frac{3}{2}}\gamma_5^{\frac{3}{2}}P(X),
\end{equation*}
which transforms equations \eqref{hamiltonian_1} and \eqref{hamiltonian_2} into
\begin{align}
Q_X&=P+\mathcal{O}(\delta^\frac{1}{2})\label{fullsystem021},\\
P_X&=Q+\frac{3A(\rho,h)}{2}Q^2+\mathcal{O}(\delta^\frac{1}{2}),\label{fullsystem022}
\end{align}
where $A(\rho,h)=\gamma_5^\frac{3}{2}(\frac{1}{h^2}-\rho)$.
The truncated system
\begin{align}
Q_X&=P\label{hamiltonian_trunc1}\\
P_X&=Q+\frac{3A(\rho,h)}{2}Q^2\label{hamiltonian_trunc2}
\end{align}
has the solution
\begin{align*}
Q(X)&=-\frac{\text{sech}^2(\frac{X}{2})}{A(\rho,h)}\\
P(X)&=\frac{\text{sech}^2(\frac{X}{2})\tanh(\frac{X}{2})}{A(\rho,h)}
\end{align*}
which corresponds to a symmetric homoclinic orbit surrounding the equilibrium $(Q,P)=(-\frac{2}{3A(\rho,h)},0)$.
For $\delta=0$, the stable manifold $W_s^0$ corresponding to the equilibrium $(0,0)$ is equal to the homoclinic orbit. Note that $(Q(0),P(0))=(-\frac{1}{A(\rho,h)},0)$. From this we get that the tangent space at $(-\frac{1}{A(\rho,h)},0)$ of $W_s^0$ is given by $\{Q=0\}$. This implies that the homoclinic orbit intersects the set $\{P=0\}$ transversally. This will also hold for $W_s^\delta$, for $\delta$ sufficiently small, by the stable manifold theorem which states that $W_s^\delta$ depends smoothly upon $\delta$. Let $(Q_0^\delta,0)$ be the intersection between $W_s^\delta$ and the set $\{P=0\}$. We use this point as initial data for the system \eqref{fullsystem021}-\eqref{fullsystem022} and find a solution $(Q,P)$. By assumption, $Q_0^\delta$ belongs to the stable manifold of $0$, which means that $(Q(X),P(X))\rightarrow (0,0)$ as $X\rightarrow \infty$. Moreover, the system \eqref{fullsystem021}-\eqref{fullsystem022} is invariant under the transformation $(Q(X),P(X))\rightarrow (Q(-X),-P(-X))$. It follows that $(Q(-X),-P(-X))$ is a solution of \eqref{fullsystem021}-\eqref{fullsystem022} as well, with the same initial data. By uniqueness these solutions must coincide. Also note that $(Q(-X),-P(-X))\rightarrow 0$ as $X\rightarrow -\infty$. It follows that for $\delta$ sufficiently small we can find a unique homoclinic solution of \eqref{fullsystem021}-\eqref{fullsystem022}.
In the original variables the corresponding wave profile $\eta$, is given by
\begin{equation*}
\eta(x)=-\frac{h_1\delta \text{sech}^2(\frac{\sqrt{\delta} x}{2h_1\sqrt{\gamma_1}})}{\frac{1}{h^2}-\rho}+\mathcal{O}(\delta^\frac{3}{2}).
\end{equation*}
Note that this is a wave of depression when $\rho<\frac{1}{h^2}$ and a wave of elevation when $\rho>\frac{1}{h^2}$. This agrees with the results found in \cite{KM}.

Let us now consider the case when $\rho-\frac{1}{h^2}$ is small. Similar to the real 1:1 resonance we let $\rho-1/h^2=\kappa\sqrt{\delta}$, where $\kappa$ is a constant and calculate the coefficient $c_4$ of the $q^4$ term in the Taylor expansion of the reduced Hamiltonian. Using the same methods as in the two previous sections, we find that
\begin{equation*}
c_{40}^0=\frac{\Psi(F_3^0[f_1,f_1,f_1],f_1])}{4}=\frac{\rho+\frac{1}{h^3}\mathcal{O}(\sqrt{\delta})}{\gamma_5^2}
\end{equation*}
which implies that
\begin{equation*}
\tilde{H}^\mu(q,p)=\frac{p^2}{2}-\frac{\delta q^2}{2\gamma_5}+\frac{\kappa\delta^2q^3}{2\gamma_5^\frac{3}{2}}+\frac{(\rho+\frac{1}{h^3}+\mathcal{O}(\sqrt{\delta}))q^4}{\gamma_5^2}+\mathcal{O}(\abs{p}^2\abs{(\delta,q,p)})+\mathcal{O}(\abs{(q,p)}^2\abs{(\delta,q,p)}^2).
\end{equation*}
We introduce the scaled variables
\begin{equation*}
X=\frac{\sqrt{\delta}x}{\sqrt{\gamma_5}},\ q(x)=\frac{\sqrt{\delta}Q(X)}{\sqrt{\gamma_5}},\ p(x)=\frac{\delta P(X)}{\gamma_5}.
\end{equation*}
Hamilton's equations are then given by
\begin{align}
Q_X&=P+\mathcal{O}(\delta^\frac{1}{2}),\label{fullsystemo41}\\
P_X&=Q+\frac{3\kappa Q^2}{2\gamma_5}-\frac{4(\rho+\frac{1}{h^3})Q^3}{\gamma_5^2}+\mathcal{O}(\delta^\frac{1}{2}).\label{fullsystemo42}
\end{align}
We find the following solutions of the truncated system:
\begin{align*}
Q_\pm(X)&=\frac{2\gamma_5}{\pm\sqrt{\kappa^2+8(\rho+\frac{1}{h^3})}\cosh(X)+\kappa},\\
P_\pm(X)&=\frac{\mp2\gamma_5\sqrt{\kappa^2+8(\rho+\frac{1}{h^3})}\sinh(X)}{\left(\pm\sqrt{k^2+8(\rho+\frac{1}{h^2})}\cosh(X)+\kappa\right)^2}.
\end{align*}
That there exist symmetric homoclinic solutions of the system \eqref{fullsystemo41}-\eqref{fullsystemo42}, can be shown in the same way as in the non-critical case. In original variables the wave profiles are given by
\begin{equation*}
\eta_\pm(x)=\frac{2h_1\sqrt{\delta}}{\pm\sqrt{\kappa^2+8(\rho+\frac{1}{h^3})}\cosh(\frac{\sqrt{\delta}x}{\sqrt{\gamma_1}h_1})+\kappa}+\mathcal{O}(\delta^\frac{3}{2}),
\end{equation*}
which agrees with the results found in \cite{KM}.
\begin{appendices}
\section{Values of the coefficients $c_1$ and $c_3$}\label{c1c3}
From \eqref{coeffc1} and \eqref{coeffc3} we obtain
\begin{equation*}
c_1=-\frac{1}{\gamma_1},
\end{equation*}
\begin{align*}
c_3(k)&=-\frac{1}{2\gamma_1^2}\Bigg\{-\frac{1}{2}\bigg[\alpha+4\beta s^2-\frac{2s}{\tanh{2hs}}-\frac{2\rho s}{\tanh(2s)}\bigg]^{-1}\bigg[\frac{s^2}{\tanh^2(hs)}+\frac{4s^2}{\tanh(hs)\tanh(2hs)}-3s^2\\
&\quad -\rho\bigg(\frac{s^2}{\tanh^2(s)}+\frac{4s^2}{\tanh(s)\tanh(2s)}-3s^2\bigg)\bigg]^2+\bigg[\rho+\frac{1}{h}-\alpha\bigg]^{-1}\bigg[\frac{s^2}{\sinh^2(hs)}+\frac{2s}{h\tanh(hs)}\\
&\quad -\rho\bigg(\frac{s^2}{\sinh^2(s)}+\frac{2s}{\tanh(s)}\bigg)\bigg]^2+\frac{6s^3}{\tanh(hs)}-\frac{4s^3}{\tanh^2(hs)\tanh(2hs)}-\frac{4s^2}{h\tanh^2(hs)}\\
&\quad +\rho\bigg(\frac{6s^3}{\tanh(s)}-\frac{4s^3}{\tanh^2(s)\tanh(2s)}-\frac{4s^2}{\tanh^2(s)}\bigg)-\frac{3s^4\beta}{2}\Bigg\}.
\end{align*}
The Laurent series expansion of $2\gamma_1^2 c_3(k)$ is given by
\begin{equation*}
2\gamma_1^2c_3(k)=\frac{a_1}{k^5}+\frac{a_2}{k^3}+\frac{a_3}{k}+\mathcal{O}(k)
\end{equation*}
where
\begin{equation*}
a_1=\frac{855(\rho-\frac{1}{h^2})^2}{(h^3+\rho)}
\end{equation*}
It is clear that $a_1\geq 0$, with equality if and only if $\rho=1/h^2$. For $\rho=1/h^2$ we have that
\begin{align*}
a_2&=0,\\
a_3&=\frac{15h^2\left(1-\frac{1}{h^2}\right)^2}{2(1+h^5}+\frac{6(1+h)}{h^3}>0.
\end{align*}
 It follows that $c_3(k)>0$ for all $\rho\in [0,1)$, $h\in(0,\infty)$ and $k$ sufficiently small. These calculations were made using Wolfram Mathematica.%se över språket här
\section{Transversality}\label{B}
Consider the equation
\begin{equation}
\partial_x^4u+P\partial_x^2u+u-u^3=0\label{maineq}.
\end{equation}
Just as in \cite{BCT} and as in section \ref{11} this equation can be seen as Hamilton's equations for a Hamiltonian $H$.
According to \cite{G1} there exists a homoclinic solution $u\in C^\infty(\mathbb{R})$ of \eqref{maineq} for $P<2$. We want to show that the stable and unstable manifolds corresponding to the equilibrium solution $0$ of \eqref{maineq} intersect transversally in the zero energy set obtained from $H$. We give a brief outline of how to do this. The idea is to show that theorems 2.1--2.4 in \cite{BCT} hold for \eqref{maineq} as well. theorems 2.1,2.3 and 2.4 are proved in the same way as in \cite{BCT}. Using Corollary 3 from \cite{JB} we can show that for $P\leq -2$, any homoclinic solution of \eqref{maineq} is symmetric and single signed with exactly one critical point. This result corresponds to theorem 2.2 in \cite{BCT}.
We may assume without loss of generality that the maximum of a positive homoclinic solution is attained at $0$, so $u$ is an even function. This implies in particular that $\partial_x^3u(0)=0$. If we again use the fact that homoclinic solutions has exactly one critical point, we get
\begin{align*}
\partial_xu(x)&<0 \text{ for all }x\in(0,\infty)\\
\partial_xu(x)&>0\text{ for all } x\in (-\infty,0).
\end{align*}
Next we prove an inequality which corresponds to (2.2) on p. 234 in \cite{BCT}.
\begin{lemma}
A positive homclinic solution $u$ of \eqref{maineq} satisfies
\begin{equation}
\frac{u^3(0)-u(0)}{-\partial_x^2u(0)}= 1+l\label{ineq1},
\end{equation}
where $l>0$.
\end{lemma}
\begin{proof}
As in \cite{BCT} we define $p\in (0,1]$ such that
\begin{equation*}
p+p^{-1}=-P
\end{equation*}
 Consider the function 
\begin{equation*}
h(x):=\partial_x^3u(x)-p \partial _xu(x),
\end{equation*}
and note that $h$ satisfies the equation
\begin{equation*}
\partial_x^2 h(x)-p^{-1}h(x)=3\partial_xu(x)u(x)^2.
\end{equation*}
Since 
\begin{align*}
&\partial_x u(x)> 0, \text{ for all } x\in (-\infty,0)\\
&h(0)=\lim_{x\rightarrow -\infty}h(x)=0,
\end{align*}
 it follows that $h(x)<0$, for all $x\in (-\infty,0]$.  The Hopf lemma then implies that $\partial_x h(0)>0$, that is
\begin{equation}
\partial_x^4u(0)-p \partial _x^2u(0)>0\label{appineq1}.
\end{equation}
Next we let $g(x)=\partial_xu(x)$. By definition
\begin{equation*}
\partial_x^2g(x)-pg(x)=h(x)< 0,\quad \text{ for all }x\in (-\infty ,0].
\end{equation*}
Again, using the Hopf lemma, we may conclude that 
\begin{equation}
\partial_x^2u(0)<0.\label{appineq2}
\end{equation}
From \eqref{appineq1}, we get
\begin{align*}
u(0)^3-u(0)&=\partial_x^4u(0)+P\partial_x^2u(0)\\
&=\partial_x^4u(0)-p\partial_x^2u(0)-p^{-1}\partial_x^2u(0)\\
&\geq -p^{-1}\partial_x^2u(0),
\end{align*}
which, together with \eqref{appineq2}, implies that
\begin{equation*}
\frac{u(0)^3-u(0)}{-\partial_x^2u(0)}>p^{-1},
\end{equation*}
and since $p^{-1}\geq 1$ we obtain \eqref{ineq1}.
\end{proof}
That the stable and unstable manifolds intersect transversally is now shown as in \cite{BCT}, p. 235--236.
\end{appendices}
\bigskip

{\bf Acknowledgments.} The author was supported by Grant No. 621-2012-3753 from the Swedish Research Council.

The author would also like to thank Erik Wahlén and Mark Groves for their help and advice when writing this article.
\normalsize
\bibliographystyle{plain}
\bibliography{biblio}

\begin{thebibliography}{10}

\bibitem{AK89}
C.~J. Amick and K.~Kirchg\"{a}ssner.
\newblock {A theory of solitary water-waves in the presence of surface
  tension}.
\newblock {\em Arch. Ration. Mech. Anal.}, 105(1), 1989.

\bibitem{AT}
C.~J. Amick and R.~E.~L. Turner.
\newblock {A global theory of internal solitary waves in two-fluid systems}.
\newblock {\em Trans. Am. Math. Soc.}, 298(2):431--484, 1986.

\bibitem{AT2}
C.~J. Amick and R.~E.~L. Turner.
\newblock {Small internal waves in two-fluid systems}.
\newblock {\em Arch. Ration. Mech. Anal.}, 108(2):111--139, 1989.

\bibitem{BCT}
B.~Buffoni, A.~R. Champneys, and J.~F. Toland.
\newblock {Bifurcation and coalescence of a plethora of homoclinic orbits for a
  Hamiltonian system}.
\newblock {\em J. Dyn. Differ. Equations}, 8(2):221--279, 1996.

\bibitem{BG}
B.~Buffoni and M.~D. Groves.
\newblock {A Multiplicity Result for Solitary Gravity-Capillary Waves in Deep
  Water via Critical-Point Theory}.
\newblock {\em Arch. Ration. Mech. Anal.}, 146(3):183--220, 1999.

\bibitem{BGT}
B.~Buffoni, M.~D. Groves, and J.~F. Toland.
\newblock {A plethora of solitary gravity-capillary water waves with nearly
  critical Bond and Froude numbers}.
\newblock {\em Philos. Trans. R. Soc. London. Ser. A. Math. Phys. Sci. Eng.},
  354:575--607, 1996.

\bibitem{DV}
R.~L. Devaney.
\newblock {Homoclinic orbits in Hamiltonian systems}.
\newblock {\em J. Differ. Equ.}, 21(2):431--438, 1976.

\bibitem{DI2}
F.~Dias and G.~Iooss.
\newblock {Capillary-gravity solitary waves with damped oscillations}.
\newblock {\em Phys. D}, 65:399--423, 1993.

\bibitem{DI}
F.~Dias and G.~Iooss.
\newblock {Capillary-gravity interfacial waves in infinite depth}.
\newblock {\em Eur. J. Mech.}, 15(3):367--393, 1996.

\bibitem{Dias2003}
F.~Dias and G.~Iooss.
\newblock {Water-Waves as a Spatial Dynamical System}.
\newblock {\em Handb. Math. Fluid Dyn.}, 2:443--499, 2003.

\bibitem{EL}
C.~Elphick.
\newblock {Global Aspects of Hamiltonian Normal Forms}.
\newblock {\em Phys. Lett. A}, 127(8,9):418--424, 1988.

\bibitem{G1}
M.~D. Groves.
\newblock {Solitary-wave solutions to a class of fifth-order model equations}.
\newblock {\em Nonlinearity}, 11(2):341--353, 1999.

\bibitem{GW07}
M.~D. Groves and E.~Wahl\'{e}n.
\newblock {Spatial dynamics methods for solitary gravity-capillary water waves
  with an arbitrary distribution of vorticity}.
\newblock {\em SIAM J. Math. Anal.}, 39(3):932--964, 2007.

\bibitem{GW08}
M.~D. Groves and E.~Wahl\'{e}n.
\newblock {Small-amplitude Stokes and solitary gravity water waves with an
  arbitrary distribution of vorticity}.
\newblock {\em Phys. D Nonlinear Phenom.}, 237:1530--1538, 2008.

\bibitem{IK}
G.~Iooss and K.~Kirchg\"{a}ssner.
\newblock {Bifurcation of solitary waves subject to low surface tension}.
\newblock {\em Comptes Rendus l'Acad\'{e}mie des Sci. S\'{e}rie I.
  Math\'{e}matique}, 311(5):265--268, 1990.

\bibitem{IK2}
G.~Iooss and K.~Kirchg\"{a}ssner.
\newblock {Water waves for small surface tension: an approach via normal form}.
\newblock {\em Proc. R. Soc. Edinburgh. Sect. A. Math.}, 122:267--299, 1992.

\bibitem{IP}
G.~Iooss and M.~P\'{e}rou\`{e}me.
\newblock {Perturbed homoclinic solutions in reversible 1:1 resonance vector
  fields}.
\newblock {\em J. Differ. Equ.}, 102(1):62--88, 1993.

\bibitem{JA}
C.~R. Jackson.
\newblock {An atlas of internal solitary-like waves and their properties
  http://www.internalwaveatlas.com/Atlas2\_index.html}.

\bibitem{KA}
T.~Kato.
\newblock {\em {Perturbation theory for linear operators}}.
\newblock Springer-Verlag New York, Inc., New York, 1966.

\bibitem{K82}
K.~Kirchg\"{a}ssner.
\newblock {Wave-Solutions of Reversible Systems and Applications}.
\newblock {\em J. Differ. Equations}, 45:113--127, 1982.

\bibitem{KI}
K.~Kirchg\"{a}ssner.
\newblock {Nonlinearly resonant surface waves and homoclinic bifurcation}.
\newblock {\em Adv. Appl. Mech.}, 26:135--181, 1988.

\bibitem{KM}
P.~Kirrmann.
\newblock {\em {Reduktion nichtlinearer elliptischer systeme in
  Zylindergebeiten unter Verwendung von optimaler Regularit\"{a}t in
  H\"{o}lder-R\"{a}umen}}.
\newblock PhD thesis, Universit\"{a}t Stuttgart, 1991.

\bibitem{LD}
O.~Laget and F.~Dias.
\newblock {Numerical computation of capillary–gravity interfacial solitary
  waves}.
\newblock {\em J. Fluid Mech}, 349:221--251, 1997.

\bibitem{MH}
K.~R. Meyer, G.~R. Hall, and D.~Offin.
\newblock {\em {Introduction to Hamiltonian Dynamical Systems and the N-Body
  Problem}}.
\newblock 2009.

\bibitem{MI}
A.~Mielke.
\newblock {Reduction of quasilinear elliptic equations in cylindrical domains
  with applications}.
\newblock {\em Math. Methods Appl. Sci.}, 10:51--66, 1988.

\bibitem{Mielke1}
Alexander Mielke.
\newblock {Homoclinic and heteroclinic solutions in two-phase flow}.
\newblock {\em Adv. Ser. Nonlinear Dynam.}, 7:353--362, 1995.

\bibitem{SA}
R.~L. Sachs.
\newblock {On the existence of small amplitude solitary waves with strong
  surface tension}.
\newblock {\em J. Differ. Equations}, 51:31--51, 1991.

\bibitem{SC}
S.~M. Sun and M.~C. Shen.
\newblock {Solitary waves in a two-layer fluid with surface tension}.
\newblock {\em SIAM J. Math. Anal.}, 24(4):866--891, 1993.

\bibitem{JB}
J.~B. van~den Berg.
\newblock {The phase-plane picture for a class of fourth-order conservative
  differential equations}.
\newblock {\em J. Differ. Equ.}, 161(1):110--153, 2000.

\end{thebibliography}
\end{document}